\documentclass[12pt]{article}

\usepackage{amssymb}
\usepackage{amsmath}
\usepackage{amsthm}
\usepackage{url} 

\allowdisplaybreaks

{\theoremstyle{plain}
 \newtheorem{theorem}{Theorem}
 \newtheorem{lemma}{Lemma}
}
{\theoremstyle{definition}
\newtheorem{example}{Example}
\newtheorem{remark}{Remark}
}

\begin{document}

\begin{center}
{\Large The construction of $q$-analogues via $_3\phi_2$-series and $q$-difference equations}

 \ 

{\textsc{John M. Campbell}}

\end{center}

\begin{abstract}
 We apply the {\tt EKHAD}-normalization method given in our recent work to obtain, via the  $q$-version of Zeilberger's algorithm, 
 $q$-WZ pairs $(\widetilde{F}, \widetilde{G})$ 
 such that $\sum_{k = 0}^{\infty} \widetilde{F}(0, k)$ may be expressed as a basic hypergeometric series of the form 
 ${}_{3}\phi_2$ with multiple free parameters, and in such a way
 so that $\sum_{k=0}^{\infty} \widetilde{F}(0, k) = \sum_{n=0}^{\infty} \widetilde{G}(n, 0)$. In 
 contrast to how previous applications of {\tt EKHAD}-normalization relied on $q$-analogues for specific WZ pairs introduced by 
 Guillera, our multiparameter approach provides a broad framework in the construction of $q$-analogues for accelerated series for 
 universal constants such as $\pi$. We apply this multiparameter version of {\tt EKHAD}-normalization to introduce and prove 
 $q$-analogues for accelerated hypergeometric series attributed to many authors, including (alphabetically) Adamchik and Wagon, 
 Ap\'{e}ry, Chu, Chu and Zhang, Fabry, Guillera, Ramanujan, and Zeilberger. 
\end{abstract}

\vspace{0.1in}


\vspace{0.1in}


\section{Introduction}
 The application of $q$-series is of importance within many areas of mathematics, with reference to the text by Gasper and Rahman on 
 basic hypergeometric series \cite{GasperRahman2004} and many subsequent and related research works. As in Gasper and Rahman's 
 monograph \cite[p.\ 36]{GasperRahman2004}, we highlight notable research works in the history of the application of $q$-difference 
 equations, as in the works (chronologically listed) of Jackson \cite{Jackson1910q}, Carmichael \cite{Carmichael1912}, Adams 
 \cite{Adams1931}, Starcher~\cite{Starcher1931}, Trjitzinsky \cite{Trjitzinsky1933}, and Andrews \cite{Andrews1968}. Since
 Wilf--Zeilberger theory marks a paradigm shift \cite{PetkovsekWilfZeilberger1996} given by the use 
 of telescoping methods in the field of 
 symbolic computation and beyond \cite{Zeilberger1991}, this motivates new developments and applications of $q$-difference 
 equations via the $q$-Zeilberger algorithm \cite[\S6]{PetkovsekWilfZeilberger1996}. 

 The famous series for $\frac{1}{\pi}$ introduced by Ramanujan in 1914 \cite[pp.\ 352--354]{Berndt1994} \cite{Ramanujan1914} may be 
 seen as prototypical instances of fast converging, hypergeometric series for $\pi$. Given that Ramanujan's series for $\frac{1}{\pi}$ were 
 introduced in 1914, and considering the importance of Ramanujan's series within number theory and other areas, it seems remarkable 
 that the first known $q$-analogues for any of Ramanujan's series were introduced as late as 2018, as in the work of Guo and 
 Liu~\cite{GuoLiu2018}, and as noted by Guo and Zudilin~\cite{GuoZudilin2018}. 
 The methodologies applied in our current paper in the derivation of 
 $q$-analogues related to Ramanujan's summations are inspired by $q$-analogues subsequent to and related to the Guo--Liu 
 $q$-analogues~\cite{ChenChu2021,ChenChu2023,Guillera2018,Guo2020,GuoZudilin2018,HouKrattenthalerSun2019,HouSun2021,
Sun2019,WangZhong2023,Wei2023,Wei2020,WeiRuan2024}. 

 The techniques in this paper build upon and extend our previous use of what we referred to as \emph{{\tt EKHAD}-normalization} 
 \cite{Campbell2026}. 
 Whereas this previous work relied on applying {\tt EKHAD}-normalization using \emph{specific} WZ pairs previously 
 introduced by Guillera \cite{Guillera2006,Guillera2002}, we have formulated a much more general way of applying 
 {\tt EKHAD}-normalization using multiparameter input functions, and this may be seen by analogy with Au's method of WZ seeds 
 \cite{Au2025} and $q$-WZ seeds \cite{Auunpublished} along with our previous work 
 \cite{Campbell2025Hypergeometric,Campbellunpublished,CampbellLevrie2024} related
 to an acceleration method due to Wilf \cite{Wilf1999}. 

 In this paper, we apply {\tt EKHAD}-normalization \cite{Campbell2026} 
 to input functions 
 $F(n$, $k$; $q)$ such that 
 $$\sum_{k=0}^{\infty} {F}(0,k;q) = {}_{3}\phi_{2}\!\!\left[ \begin{matrix} 
 q^{a}, q^{b}, q \vspace{1mm}\\ 
 q^{c}, q^{d} \end{matrix} \ \Bigg| \ q; q 
 \right],$$ 
 for suitably bounded parameters $a$, $b$, $c$, and $d$ (referring to Section \ref{prelim} below for preliminaries). 
 {\tt EKHAD}-normalization has the effect of producing a $q$-WZ pair
 $(\widetilde{F}, \widetilde{G})$, in such a way so that 
\begin{equation}\label{introsumFsumG}
 \sum_{k=0}^{\infty} \widetilde{F}(0,k;q)
 = \sum_{n=0}^{\infty} \widetilde{G}(n,0;q). 
\end{equation}
 By applying, subject to multiplication on both sides of \eqref{introsumFsumG}
 by a normalizing factor, the limiting operation given by 
 $q \to 1^{+}$ after restricting $q$ so that $q > 1$, 
 the right-hand side of \eqref{introsumFsumG}
 provides, as in the applications below, 
 an accelerated version of the left-hand series. 
 We succeed in applying this approach to 
 introduce $q$-analogues 
 of known evaluations for
 series of convergence rates $\frac{1}{4}$, $-\frac{1}{4}$, $-\frac{1}{27}$, 
 and $\frac{1}{64}$. 

\section{Preliminaries and background}\label{prelim}
 The $\Gamma$-function may be seen as among the most important special functions, with reference to Rainville's textbook on special 
 functions \cite[\S2]{Rainville1960}, and may be defined 
 via an Euler integral, with 
\begin{equation}\label{displayintegral}
 \Gamma(x) = \int_{0}^{\infty} t^{x-1} e^{-t} \, dt 
\end{equation} 
 for $\Re(x) > 0$. As a notational shorthand, we may write 
\begin{equation}\label{simplifyGamma}
 \Gamma\!\left[ \begin{matrix} \alpha, \beta, \ldots, \gamma \vspace{1mm} \\ 
 A, B, \ldots, C \end{matrix} \right] = \frac{ \Gamma(\alpha) \Gamma(\beta) 
 \cdots \Gamma(\gamma) }{ \Gamma(A) \Gamma(B) \cdots \Gamma(C)}. 
\end{equation}
 The \emph{Pochhammer symbol} may then be defined so that 
\begin{equation}\label{origPoch}
 \left( a \right)_{n} = 
 \Gamma\!\left[ \begin{matrix} a + n \vspace{1mm} \\ 
 a \end{matrix} \right], 
\end{equation}
 and, by analogy with \eqref{simplifyGamma}, we write 
\begin{equation*}
 \left[ \begin{matrix} \alpha, \beta, \ldots, \gamma \vspace{1mm} \\ 
 A, B, \ldots, C \end{matrix} \right]_{k} = \frac{ (\alpha)_{k} (\beta)_{k} 
 \cdots (\gamma)_{k} }{ (A)_{k} (B)_{k} \cdots (C)_{k}}. 
\end{equation*}
 We may define \emph{generalized hypergeometric series} so that 
\begin{equation}\label{pFq}
 {}_{r}F_{s}\!\!\left[ \begin{matrix} 
 a_{1}, a_{2}, \ldots, a_{r} \vspace{1mm}\\ 
 b_{1}, b_{2}, \ldots, b_{s} \end{matrix} \ \Bigg| \ x 
 \right] = \sum_{n=0}^{\infty} 
 \left[ \begin{matrix} a_{1}, a_{2}, \ldots, a_{r} \vspace{1mm} \\ 
 b_{1}, b_{2}, \ldots, b_{s} \end{matrix} \right]_{n} \frac{x^{n}}{n!}, 
\end{equation}
 referring to Bailey's classic text on such series for details and background material \cite{Bailey1964}. We refer to the value $x$ in 
 \eqref{pFq} as the \emph{convergence rate} of either side of \eqref{pFq}. 

 The \emph{$q$-shifted factorial}, which is also referred to as the \emph{$q$-Pochhammer 
 symbol}, is 
\begin{equation}\label{qdiscreteshift}
 (a;q)_{n} = \begin{cases} 
 1 & n = 0, \\
 (1-a)(1-aq) \cdots (1-aq^{n-1}) & n \in \mathbb{N}. 
 \end{cases}
\end{equation}

  For $n$ in $\mathbb{N}_{0}$, the $q$-shifted factorial in \eqref{qdiscreteshift} is a finite product, 
 and its definition does not require the condition $|q| < 1$. Thus, whenever a $q$-shifted factorial occurring below
 has a nonnegative integer index, it is understood according to the finite product definition in 
 \eqref{qdiscreteshift}, subject to the nonvanishing of denominator factors. 
 When $|q| < 1$, we also define 
 $$ (a;q)_{\infty} = \prod_{k=0}^{\infty} (1-aq^k). $$ 
 For a noninteger index $\nu$,
 we then define 
\begin{equation}\label{nuindex}
 (a;q)_{\nu} = \frac{ (a;q)_{\infty} }{ (aq^{\nu}; q)_{\infty} }, 
\end{equation}
 whenever the quotient in \eqref{nuindex} is defined. 
 In particular, the identities introduced in 
 Sections \ref{sectionMotivating} and \ref{multisec}
 are for $ q > 1$, and the $q$-shifted factorials appearing within the summands have nonnegative integer indices
 and are thus interpreted exclusively as finite products. 

 Whenever an expression $q^{\lambda}$ occurs with a noninteger exponent $\lambda$, 
 we fix a single branch $\log q$ and define $q^{\lambda} = \exp(\lambda \log q)$. 
 All powers of $q$ occurring within a given identity are interpreted using the same branch $\log q$. 
 If $q$ is positive and real, then $\log q$ is understood as the usual (real) logarithm. 

 We also make use of the notational shorthand such that 
\begin{equation*}
 \left[ \begin{matrix} \alpha, \beta, \ldots, \gamma \vspace{1mm} \\ 
 A, B, \ldots, C \end{matrix} \, \Bigg| \, q \right]_{n} 
 = \frac{ \left( \alpha; q \right)_{n} 
 \left( \beta; q \right)_{n} \cdots \left( \gamma; q \right)_{n} }{ 
 \left( A; q \right)_{n} 
 \left( B; q \right)_{n} \cdots \left( C; q \right)_{n}}. 
\end{equation*}
 The $q$-Pochhammer symbol gives us a $q$-analogue of the function defined in \eqref{origPoch}
 in the sense that 
 $$ \lim_{q \to 1} \frac{ \left( q^{a};q \right)_{n} }{(1-q)^{n}} = (a)_{n}. $$
 Similarly, the $q$-Gamma function may be defined so that
 $$ \Gamma_{q}(x) = (1-q)^{1-x} \left[ \begin{matrix} q \vspace{1mm} \\ 
 q^{x} \end{matrix} \, \Bigg| \, q \right]_{\infty}, $$
 for $|q| < 1$, giving us a $q$-analogue of \eqref{displayintegral}, with 
 $ \lim_{q \to 1} \Gamma_{q}(x) = \Gamma(x)$, 
 and we adopt a notational shorthand for combinations of $q$-Gamma expressions
 by analogy with \eqref{simplifyGamma}. 
 Also, the \emph{$q$-bracket} 
 is defined so that $[n]_{q} = \frac{1-q^n}{1-q}$, 
 noting that $\lim_{q \to 1} [n]_{q} = n$. 
 By then defining the $q$-factorial
 so that $[n]_{q}! = [1]_{q} [2]_{q} \cdots [n]_{q}$, we may define \emph{Gaussian binomial coefficients}
 so that 
 $$ \left[ \begin{matrix} 
 n \\ 
 k 
 \end{matrix} \right]_{q} = \frac{ [n]_{q}! }{ [k]_{q}! [n-k]_{q}! } $$ for $ k \leq n$. 

 \emph{Unilateral basic hypergeometric series} provide a $q$-analogue of \eqref{pFq} and may be defined so that 
\begin{multline*}
 {}_{j}\phi_{k}\!\!\left[ \begin{matrix} 
 a_{1}, a_{2}, \ldots, a_{j} \vspace{1mm}\\ 
 b_{1}, b_{2}, \ldots, b_{k} \end{matrix} \ \Bigg| \ q; z 
 \right] \\ 
 = \sum_{n=0}^{\infty} \frac{z^n}{ (q;q)_{n} } 
 \left[ \begin{matrix} 
 a_{1}, a_{2}, \ldots, a_{j} \vspace{1mm} \\ 
 b_{1}, b_{2}, \ldots, b_{k} \end{matrix} \, \Bigg| \, q \right]_{n} 
 \left( (-1)^{n} q^{\binom{n}{2}} \right)^{1 + k - j}. 
\end{multline*}

\subsection{{\tt EKHAD}-normalization}\label{sectionEKHAD}
 A bivariate function $F(n, k)$ is \emph{hypergeometric} if $\frac{F(n+1,k)}{F(n,k)}$ and $\frac{F(n,k+1)}{F(n,k)}$ are both rational functions, 
 and the coefficients are taken to be in $\mathbb{Q}$ in the context of WZ theory. For a rational function $R(n, k)$, write $G(n, k) = R(n, 
 k) F(n, k)$, and suppose that $F$ and $G$ satisfy a difference equation of the form 
\begin{equation}\label{fundamentalWZ}
 F(n+1, k) - F(n, k) = G(n, k+1) - G(n, k). 
\end{equation}
 This may be seen as the fundamental difference equation 
 in WZ theory, with reference to the classic text by 
 Petkov\v sek et al.\ on WZ theory \cite{PetkovsekWilfZeilberger1996}. 
 While a given function $F(n, k)$ may not be such that there exists a companion $G(n, k)$
 of the desired form satisfying \eqref{fundamentalWZ}, 
 \emph{Zeilberger's algorithm} \cite[\S6]{PetkovsekWilfZeilberger1996} 
 may be thought of as extending the WZ difference equation in \eqref{fundamentalWZ}
 in such a way so that this algorithm
 is applicable to \emph{any} hypergeometric input function $F(n, k)$, 
 referring to the same text by Petkov\v sek et al.\ for details. 
 Our techniques rely on the $q$-version of Zeilberger's algorithm, 
 again referring to the appropriate text by Petkov\v sek et al.\ for details. 

 Let $F(n,k;q)$ denote a bivariate $q$-hypergeometric function. 
 Suppose that the $q$-Zeilberger algorithm produces a first-order $q$-difference equation of the form 
 \begin{equation}\label{generalqZeil}
 p_1(n;q) F(n+1,k;q) + p_2(n;q) F(n,k;q) 
 = H(n,k+1;q) - H(n,k;q), 
\end{equation} 
 where $p_1(n;q)$ and $p_2(n;q)$
 are independent of $k$, 
 and where $H(n,k;q)$ is the $q$-hypergeometric companion to 
 $F(n,k;q)$ produced by Zeilberger's algorithm (noting that we are avoiding the notation $G(n,k;q)$
 to avoid confusion with a normalized variant of $H$ given below). 
 In what follows, the expressions $F$ and $H$ denote ``unnormalized'' 
 hypergeometric functions, 
 whereas $\widetilde{F}$ and $\widetilde{G}$ denote the normalized $q$-WZ pair obtained from $F$ and $H$. 

\begin{remark}
The coefficient $+1$ associated with the term 
 $p_2(n;q) F(n,k;q)$ in \eqref{generalqZeil} may be seen in contrast to the coefficient 
 $-1$ associated with the term $F(n, k)$ in \eqref{fundamentalWZ}. 
 The suggested convention concerning the specified positive coefficient in 
 \eqref{generalqZeil} is in accordance with the original formulation of {\tt EKHAD}-normalization \cite{Campbell2026}. 
\end{remark}

 For the purposes of our paper, {\tt EKHAD}-normalization may be summarized with the following Lemma, 
 and this agrees with the original formulation of {\tt EKHAD}-normalization \cite{Campbell2026}. 

\begin{lemma}\label{lemmaEKHAD}
 Fix an integer $j \geq 0$. 
 Let $F(n,k;q)$ and $H(n,k;q)$ be $q$-hypergeo-metric functions, 
 and suppose that these functions satisfy
 the $q$-difference equation in \eqref{generalqZeil} for all integers $n \geq j$. 
 Also, assume that $p_2(n;q) \neq 0$ 
 for each integer $n \geq j$. 
 Define 
\begin{equation}\label{definehnq}
 h(n;q) = (-1)^{n} \prod_{i=j}^{n-1} \frac{p_1(i;q)}{p_2(i;q)}. 
\end{equation}
 Here and throughout, if $n=j$, then the product on the right of \eqref{definehnq} 
 is understood to be an empty product and is thus equal to $1$ in this case. 
 Now, let 
\begin{equation}\label{tildeFnkq} 
 \widetilde{F}(n,k;q) = h(n;q) F(n,k;q) 
\end{equation}
 and 
\begin{equation}\label{tildeGnkq}
 \widetilde{G}(n,k;q) = -\frac{h(n;q)}{p_2(n;q)} H(n,k;q). 
\end{equation}
 Then the functions in \eqref{tildeFnkq} and \eqref{tildeGnkq}
 are such that the $q$-WZ difference equation
\begin{equation}\label{qWZequation}
 \widetilde{F}(n+1,k;q) - \widetilde{F}(n,k;q)
 = \widetilde{G}(n,k+1;q) - \widetilde{G}(n,k;q) 
\end{equation}
 holds. 
\end{lemma}

\begin{proof}
 From \eqref{definehnq}, we find that $h(n;q)$ satisfies the recurrence relation 
\begin{equation}\label{hrecurrence} 
 h(n+1;q) = -\frac{p_1(n;q)}{p_2(n;q)} h(n;q). 
\end{equation}
 From the recurrence in \eqref{hrecurrence}, we proceed to rewrite the $q$-difference equation 
 $$ \widetilde{F}(n+1, k;q) - \widetilde{F}(n,k;q)
 = h(n+1;q) F(n+1,k;q) - h(n;q) F(n,k;q) $$
 so that 
\begin{multline*} 
 \widetilde{F}(n+1,k;q) - \widetilde{F}(n,k;q)
 = \\ -\frac{h(n;q)}{p_2(n;q)}
 \big( p_1(n;q) F(n+1,k;q) + p_2(n;q) F(n,k;q) \big). 
\end{multline*}
 From this and the $q$-difference equation in \eqref{generalqZeil}, we find that 
\begin{equation}\label{replacewGtilde} 
 \widetilde{F}(n+1,k;q) 
 - \widetilde{F}(n,k;q) 
 = -\frac{h(n;q)}{p_2(n;q)}
 \big( H(n,k+1;q) - H(n,k;q) \big). 
\end{equation}
 From the definition in \eqref{tildeGnkq} and the $q$-difference equation in \eqref{replacewGtilde}, 
 we obtain \eqref{qWZequation}. 
\end{proof}

\begin{remark}
 The integer $j$ is the base index from which the normalization recurrence is applied. 
 Given a $q$-hypergeometric input function 
 to be normalized according to Lemma \ref{lemmaEKHAD}, 
 as in \eqref{tildeFnkq}, the index $j$ is chosen so that 
\begin{equation}\label{p2iqneq0} 
 p_2(i;q) \neq 0
\end{equation} 
for each integer $i \geq j$. 
 For practical purposes, this index $j$ is typically selected as $0$ or $1$, 
 in the context of our past \cite{Campbell2026} and current applications of {\tt EKHAD}-normalization. 
 In any case, wherever possible, the index $j$ is selected as the least nonnegative integer such that 
 \eqref{p2iqneq0} holds for $i \geq j$. 
\end{remark}

 Our past work \cite{CampbellLevrie2024} concerned multiparameter versions of an acceleration method due to Wilf 
 \cite{Wilf1999}. 
 In this past work \cite{CampbellLevrie2024}, by applying Zeilberger's algorithm
 to certain input functions such as 
 \begin{equation}\label{freelowerquarter}
 F(n, k) := \left[ \begin{matrix} 
 a, b \vspace{1mm} \\ 
 n, n
 \end{matrix} \right]_{k}, 
\end{equation}
 this produced inhomogeneous, first-order difference equations that were \emph{not}
 normalized, in accordance with Wilf's method \cite{Wilf1999}. 
 Due to certain conver-gence-related issues, 
 this approach does not appear to be applicable in general using $q$-analogues of input function as in 
 \eqref{freelowerquarter}, 
 and this leads us to apply {\tt EKHAD}-normalization to $q$-analogues of 
 hypergeometric expressions as in \eqref{freelowerquarter}, 
 and, instead of using Wilf's acceleration method (which relies on inhomogeneous difference equations, 
 in contrast to \eqref{fundamentalWZ}), we instead consider a telescoping argument 
 based on WZ difference equations. 

\subsection{A telescoping lemma}
 To pass from the normalized $q$-WZ relation to an identity between infinite series, 
 we make use of the following lemma. 

\begin{lemma}\label{lemmaRectangularTelescope}
 Let $q$ be fixed. Let 
\begin{equation}\label{FGcolon}
 \widetilde{F}, \widetilde{G}\colon \mathbb{N}_{0} \times \mathbb{N}_{0} \to \mathbb{C}
\end{equation}
 satisfy the $q$-WZ difference equation in \eqref{qWZequation} 
 for $n, k \in \mathbb{N}_{0}$, 
 and with the understanding that any auxiliary parameters associated with the functions in 
 \eqref{FGcolon} are fixed. 
 Also, assume that all of the following conditions hold. 

 \ 

\noindent $\bullet$ For each fixed $k \in \mathbb{N}_{0}$, 
 the vanishing relation 
\begin{equation}\label{hypothesis1telescope}
 \lim_{N \to \infty} \widetilde{F}(N, k;q) = 0 
\end{equation}
 holds. 

 \ 

\noindent $\bullet$ For each fixed $k \in \mathbb{N}_{0}$, 
 the series
\begin{equation}\label{hypothesis2telescope}
 \sum_{n=0}^{\infty} \left| \widetilde{G}(n,k;q) \right| 
\end{equation}
 converges. 

 \ 

\noindent $\bullet$ The series 
\begin{equation}\label{hypothesis3telescope}
 \sum_{k=0}^{\infty} \left| \widetilde{F}(0, k; q) \right| 
\end{equation}
 converges. 

 \ 

\noindent $\bullet$ There exists an integer $k_0 \geq 0$, constants $M > 0$ and $\rho \in (0, 1)$, 
 and a sequence of nonnegative real numbers
 $(A_n)_{n \geq 0}$ satisfying $ \sum_{n=0}^{\infty} A_{n} < \infty $ and such that 
 $ \left| \widetilde{G}(n,k;q) \right| \leq M \rho^{k - k_0} A_n $
 holds for all $n \geq 0$ and for all $k \geq k_0$. 

 \ 
 
\noindent Then
 $$ \sum_{k=0}^{\infty} \widetilde{F}(0,k;q)
 = \sum_{n=0}^{\infty} \widetilde{G}(n,0;q). $$
 Furthermore, the vanishing relation 
 $$ \lim_{k \to \infty} \sum_{n=0}^{\infty} \left| \widetilde{G}(n,k;q) \right| = 0 $$
 holds. 
 \end{lemma}

\begin{proof}
 Let $N \geq 1$ and $K \geq 1$. 
 By summing the $q$-WZ relation in \eqref{qWZequation}
 over the finite rectangle given by 
 the relations among
 $0 \leq n \leq N-1$ and $0 \leq k \leq K-1$, we obtain that 
\begin{multline*}
 \sum_{n=0}^{N-1} \sum_{k=0}^{K-1} 
 \big( \widetilde{F}(n+1,k;q) 
 - \widetilde{F}(n,k;q) \big) 
 = \\ \sum_{n=0}^{N-1} \sum_{k=0}^{K-1}
 \big( \widetilde{G}(n,k+1;q) - \widetilde{G}(n,k;q) \big). 
\end{multline*}
 Since both of the double sums displayed above 
 are finite, we are permitted to interchange the order of the applications
 of the given finite summation operators. 
 Telescoping with respect to $n$ on the left and with respect to $k$ on the right gives us that 
\begin{equation}\label{twofoldtelescope} 
 \sum_{k=0}^{K-1} \big( \widetilde{F}(N,k;q) - \widetilde{F}(0,k;q) \big) 
 = \sum_{n=0}^{N-1} \big( \widetilde{G}(n,K;q) - \widetilde{G}(n,0;q) \big), 
\end{equation}
 and we proceed to rewrite \eqref{twofoldtelescope} so that 
\begin{equation}\label{splitsums} 
 \sum_{k=0}^{K-1} \widetilde{F}(N,k;q) - \sum_{k=0}^{K-1} \widetilde{F}(0,k;q)
 = \sum_{n=0}^{N-1} \widetilde{G}(n,K;q) - \sum_{n=0}^{N-1} \widetilde{G}(n,0;q). 
\end{equation}
 We proceed to fix $K$ and to let $N \to \infty$. 
 Each sum over $k$ in 
 \eqref{splitsums} is finite, so that
 the first hypothesis above (which is associated with the vanishing in \eqref{hypothesis1telescope}) 
 gives us that 
 $$ \lim_{N \to \infty} \sum_{k=0}^{K-1} \widetilde{F}(N,k;q) 
 = \sum_{k=0}^{K-1} \lim_{N \to \infty} \widetilde{F}(N,k;q) = 0. $$ 
 From the second hypothesis (which is associated with the convergence of the series in \eqref{hypothesis2telescope}), 
 we have that both
 $\sum_{n=0}^{\infty} \widetilde{G}(n,K;q)$ and $\sum_{n=0}^{\infty} \widetilde{G}(n,0;q)$
 converge absolutely. Consequently, by letting $N \to \infty$ in the given rectangular region, 
 we obtain the relation 
\begin{equation}\label{fromabsolute} 
 -\sum_{k=0}^{K-1} \widetilde{F}(0,k;q) 
 = \sum_{n=0}^{\infty} \widetilde{G}(n, K;q) - \sum_{n=0}^{\infty} \widetilde{G}(n,0;q), 
\end{equation}
 and we rewrite 
 \eqref{fromabsolute} so that 
\begin{equation}\label{rewritesign} 
 \sum_{k=0}^{K-1} \widetilde{F}(0,k;q)
 = \sum_{n=0}^{\infty} \widetilde{G}(n,0;q) - \sum_{n=0}^{\infty} \widetilde{G}(n,K;q). 
\end{equation}
 It remains to show that the final series in \eqref{rewritesign}
 tends to $0$ as $K \to \infty$. 

 Now, for $K \geq k_0$ the final hypothesis gives us that 
\begin{align}
\begin{split}
 \left| \sum_{n=0}^{\infty} \widetilde{G}(n,K;q) \right|
 & \leq \sum_{n=0}^{\infty} \left| \widetilde{G}(n,K;q) \right| \\ 
 & \leq M \, \rho^{K - k_0} \sum_{n=0}^{\infty} A_n.
\end{split}\label{splittelescope}
\end{align}
 Now, since 
 $0 < \rho < 1$ and $\sum_{n=0}^{\infty} A_n < \infty$, 
 we find that 
\begin{equation}\label{miscvanish} 
 \lim_{K \to \infty} M \, \rho^{K - k_0} \sum_{n=0}^{\infty} A_n = 0. 
\end{equation}
 From \eqref{splittelescope} and \eqref{miscvanish} together, we see that 
\begin{equation}\label{absseriesvanish}
 \lim_{K \to \infty} \sum_{n=0}^{\infty} \left| \widetilde{G}(n,K;q) \right| = 0. 
\end{equation}
 In turn, the vanishing in \eqref{absseriesvanish} implies that 
\begin{equation}\label{tildeGsumvanish} 
 \lim_{K \to \infty} \sum_{n=0}^{\infty} \widetilde{G}(n,K;q) = 0. 
\end{equation}
 Finally, the third hypothesis (recalling \eqref{hypothesis3telescope})
 guarantees that 
\begin{equation}\label{limsumsum} 
 \lim_{K \to \infty} \sum_{k=0}^{K-1} \widetilde{F}(0,k;q) = \sum_{k=0}^{\infty} \widetilde{F}(0,k;q).
\end{equation}
 By then letting $K \to \infty$ in \eqref{rewritesign}, 
 a combined application of \eqref{tildeGsumvanish} and \eqref{limsumsum} gives us the desired result. 
\end{proof}

\subsection{Ramanujan-type formulas and $\Gamma$-values} 
 Ramanujan's $17$ formulas for $\frac{1}{\pi}$ \cite[pp.\ 352--354]{Berndt1994} \cite{Ramanujan1914} include 
\begin{equation}\label{Ramanujanquarter}
 \frac{4}{\pi} = 
 \sum_{n = 0}^{\infty} 
 \left( \frac{1}{4} \right)^{n} \left[ \begin{matrix} 
 \frac{1}{2}, \frac{1}{2}, \frac{1}{2} 
 \vspace{1mm} \\ 
 1, 1, 1
 \end{matrix} \right]_{n} (6 n + 1), 
 \end{equation}
 and one of the main results in our paper is given by a 
 $q$-analogue of \eqref{Ramanujanquarter} 
 that we highlight in Section \ref{subRamanujan}
 and that we prove as a consequence of Theorem \ref{theoremquarter} below. 
 Further formulas, out of Ramanujan's $17$ formulas for $\frac{1}{\pi}$, include
 $$ \frac{27}{4\pi} = \sum_{n = 0}^{\infty} 
 \left( \frac{2}{27} \right)^{n} \left[ \begin{matrix} 
 \frac{1}{3}, \frac{1}{2}, \frac{2}{3} 
 \vspace{1mm} \\ 
 1, 1, 1
 \end{matrix} \right]_{n} (15 n + 2) 
 $$ and 
 $$ \frac{72}{\pi} = \sum_{n = 0}^{\infty} 
 \left( -\frac{1}{324} \right)^{n} \left[ \begin{matrix} 
 \frac{1}{4}, \frac{1}{2}, \frac{3}{4} 
 \vspace{1mm} \\ 
 1, 1, 1
 \end{matrix} \right]_{n} (260 n + 23). $$

 The importance of Ramanujan's $\pi$-formulas in the history of mathematics, especially in regard to subjects such as computational 
 number theory, is made clear through the expository work 
 due to Baruah et al.\ \cite{BaruahBerndtChan2009}
 along with the classic monograph on $\pi$ and the AGM \cite{BorweinBorwein1987}. 
 This motivates the concept of a 
 \emph{Ramanujan-type formula}, which, for our purposes, refers to 
 an evaluation, in closed form or in terms of constants such as $\Gamma$-values with rational arguments, 
 for a series of the form 
\begin{equation}\label{20260213945PM12AS}
 \sum_{n=0}^{\infty} z^n \frac{ (a_{1})_{n} (a_{2})_{n} (a_{3})_{n} }{ 
 (b_{1})_{n} (b_{2})_{n} (b_{3})_{n} } (c_{1} n + c_2), 
\end{equation}
 with $a_i$ and $b_i$ and $c_i$ being rational, 
 and with the condition that 
 the difference between the arguments of any lower 
 and upper Pochhammer symbols within the summand in \eqref{20260213945PM12AS} is not an integer. 
 Many of our $q$-analogues introduced in this paper
 are $q$-analogues of known Ramanujan-type formulas. 
 Moreover, our techniques have the added benefit of producing
 many further Ramanujan-type formulas that have not previously been known. 

 While a number of our past works have concerned closed-form evaluations for series satisfying
 the given conditions for \eqref{20260213945PM12AS}, 
 there is relatively little known when it comes to the expression of 
 series satisfying the conditions for 
 \eqref{20260213945PM12AS} with $\Gamma$-values with rational arguments. 
 We apply our techniques given in this paper 
 to produce many evaluations of this form, 
 and this is motivated by past and remarkable evaluations of this form 
 due to Chen and Chu \cite{ChenChu2024}, including 
\begin{align*}
 \frac{\pi \Gamma^{2}\left( \frac{1}{3} \right)}{ 6 \Gamma^{2}\left( \frac{5}{6} 
 \right) } & = \sum_{n = 0}^{\infty} 
 \left( -\frac{1}{27} \right)^{n} \left[ \begin{matrix} 
 \frac{1}{2}, \frac{2}{3}, 1 
 \vspace{1mm} \\ 
 \frac{4}{3}, \frac{4}{3}, \frac{4}{3}
 \end{matrix} \right]_{n} (7n+3), \\ 
 \frac{2\pi^2}{\Gamma^{3}\left( \frac{2}{3} \right)} & = \sum_{n = 0}^{\infty} 
 \left( -\frac{1}{27} \right)^{n} \left[ \begin{matrix} 
 \frac{1}{6}, \frac{2}{3}, \frac{2}{3} 
 \vspace{1mm} \\ 
 1, \frac{7}{6}, \frac{4}{3}
 \end{matrix} \right]_{n} (21n+8), \ \text{and} \\ 
 \frac{2\pi^2}{\Gamma^{3}\left( \frac{1}{3} \right)} & = \sum_{n = 0}^{\infty} 
 \left( -\frac{1}{27} \right)^{n} \left[ \begin{matrix} 
 -\frac{1}{6}, \frac{1}{3}, \frac{1}{3} 
 \vspace{1mm} \\ 
 \frac{2}{3}, \frac{5}{6}, 1
 \end{matrix} \right]_{n} (21n+1). 
\end{align*}
 Our $q$-analogues of Ramanujan-type formulas for $\Gamma$-values with rational 
 arguments are also inspired by $q$-analogues of this form due to 
 Chen and Chu~\cite{ChenChu2023}, 
 with reference to their applications of 
 Carlitz inversions of
 the $q$-Pfaff-Saalsch\"{u}tz theorem. 
 This approach led Chen and Chu to introduce evaluations such as 
 $$ \frac{\Gamma^{2}\left( \frac{1}{4} \right)}{8 \Gamma^{2}\left( \frac{3}{4} \right) } 
 = \sum_{n = 0}^{\infty} 
 \left( \frac{1}{4} \right)^{n} \left[ \begin{matrix} 
 \frac{1}{2}, \frac{1}{2}, \frac{1}{2} 
 \vspace{1mm} \\ 
 1, \frac{5}{4}, \frac{5}{4}
 \end{matrix} \right]_{n} (3n+1). $$

\subsection{Convention concerning $q$-analogues}
 In this paper, a $q$-identity $A(q) = B(q)$
 is said to provide a $q$-analogue of an identity $\mathcal{A} = \mathcal{B}$
 when there exists a function $C(q)$ and a specified one-sided limiting process
 such that 
 $\lim C(q) A(q) = \mathcal{A}$ and $\lim C(q) B(q) = \mathcal{B}$. 
 For identities stated in the region $q > 1$, 
 the classical limit is taken so that $q \to 1^{+}$. For identities stated in the region $0 < q < 1$, 
 the classical limit is taken so that $q \to 1^{-}$. 
 Whenever $A(q)$ or $B(q)$ are given by infinite series, 
 any interchange of the limit 
 $q \to 1^{\pm 1}$ with the infinite summation requires a separate dominated convergence or locally uniform
 estimate.

\section{Motivating results}\label{sectionMotivating}
 We highlight, in the current section, results obtained from our multiparameter {\tt EKHAD}-normalization method that provide 
 $q$-analogues of notable summation identities from a number of different authors. 

 Ramanujan's formulas for $\frac{1}{\pi}$ are groundbreaking in terms of the history associated with the computation of $\pi$, with 
 reference to relevant texts related to this history 
 \cite{ArndtHaenel2001,BerggrenBorweinBorwein2004,BorweinBorwein1987}. 
 Since our work concerns $q$-analogues of fast converging series for universal constants such as $\pi$, 
 the groundbreaking nature of the 
 Guo--Liu $q$-analogues introduced in 2018 
 leads us toward our $q$-analogue of the simplest or slowest converging
 out of Ramanujan's 17 series for $\frac{1}{\pi}$. 
 
\subsection{Ramanujan's series of convergence rate $\frac{1}{4}$}\label{subRamanujan}
 In Section \ref{subquarter}, we prove the relation 
\begin{multline*}
 q - 1 + \frac{q-1}{(q+1)^2} \sum_{n=0}^{\infty} 
 q^{2n} \left[ \begin{matrix} 
 q, q \vspace{1mm} \\ 
 q^4, q^4 \end{matrix} \, \Bigg| \, q^{2} \right]_{n} = \\ \sum_{n=0}^{\infty} q^{2n} 
 \left[ \begin{matrix} 
 q, q, q, q^3 \vspace{1mm} \\ 
 q^2, q^2 \end{matrix} \, \Bigg| \, q^{2} \right]_{n} 
 \frac{q^{4n+1} + q^{2n} - 2}{ (q^4;q^4)_{n} (q^6;q^4)_{n}} 
\end{multline*}
 holds for $q > 1$. By dividing both sides of the above 
 identity by $q-1$ and then letting $q \to 1^{+}$, this 
 yields the Ramanujan formula in \eqref{Ramanujanquarter}. 
 Consequently, the preceding identity is a $q$-analogue of Ramanujan's formula in this precise limiting sense. 
 For comparison, previously known $q$-analogues of \eqref{Ramanujanquarter} 
 are given in the following survey. 
 The above $q$-analogue can be shown to be equivalent to 
 a $q$-analogue recently proved by Au \cite{Auunpublished}. 
 Unless otherwise indicated, the 
 $q$-analogues in the following survey are such that $0 < q < 1$.

 The Guo--Liu $q$-analogue of \eqref{Ramanujanquarter} \cite{GuoLiu2018} is 
$$ \sum_{n=0}^{\infty} 
 q^{n^2} [6n+1] \left[ \begin{matrix} 
 q^2 \vspace{1mm} \\ 
 q^4, q^4, q^4 \end{matrix} \, \Bigg| \, q^{4} \right]_{n} (q;q^2)_{n}^{2} 
 = (1+q) 
 \left[ \begin{matrix} 
 q^{2}, q^{6} \vspace{1mm} \\ 
 q^4, q^4 \end{matrix} \, \Bigg| \, q^{4} \right]_{\infty}. $$
 The Guo $q$-analogue of \eqref{Ramanujanquarter} \cite{Guo2020} is 
\begin{multline*}
 \sum_{n=0}^{\infty} 
 q^{2n^2} [6n+1] \left[ \begin{matrix} 
 q, q, q, q \vspace{1mm} \\ 
 q^2, q^2 \end{matrix} \, \Bigg| \, q^{2} \right]_{n} \frac{1}{(q^2;q^2)_{2n}} \times \\ 
 \left( [4n+1] - \frac{q^{4n+1}[2n+1]^2}{[4n+2]} \right) = 
 \left[ \begin{matrix} 
 q, q^{3} \vspace{1mm} \\ 
 q^2, q^2 \end{matrix} \, \Bigg| \, q^{2} \right]_{\infty}. 
\end{multline*}
 Our previous $q$-analogues of \eqref{Ramanujanquarter} \cite{Campbell2026} are 
\begin{multline*}
 \sum_{n=0}^{\infty} (-1)^n 
 q^{n^2} \frac{ (1-q^{ -1 + 2n}) (2 q^{1+4n} - q^{1+2n} - 1 ) }{1 + q^{1+2n}} \\ 
 \times \left[ \begin{matrix} \frac{1}{q}, q, q, q, q \vspace{1mm} \\ 
 q^2, q^2, q^2, q^{1-2n}, q^{2+2n} \end{matrix} \, \Bigg| \, q^2 \right]_{n}
 = 
 \frac{(1-q)^2 (1+q)}{q} \left[ \begin{matrix} q, q \vspace{1mm} \\ 
 q^2, q^3 \end{matrix} \, \Bigg| \, q^2 \right]_{\frac{1}{2}}. 
\end{multline*}
 and 
\begin{multline*}
 \sum_{n=0}^{\infty} 
 q^{2 n^2} \frac{ (q - q^{2n}) (2q^{1+4n} - q^{1+2n} - 1 ) }{ (1+q^{2n}) (1 + q^{1+2n}) } \\
 \times \left[ \begin{matrix} \frac{1}{q}, q, q \vspace{1mm} \\ 
 -1, -q, q^2, q^2, q^2 \end{matrix} \, \Bigg| \, q^2 \right]_{n} 
 = 
 \frac{(1-q)^2 (1+q) }{2} \left[ \begin{matrix} q, q \vspace{1mm} \\ 
 q^2, q^3 \end{matrix} \, \Bigg| \, q^2 \right]_{\frac{1}{2}}. 
\end{multline*}
 Au's $q$-analogues \cite{Auunpublished} of \eqref{Ramanujanquarter}
 are 
$$ \sum_{n=0}^{\infty} 
 q^{2n^2} \frac{ \left( q;q^2 \right)_{n}^{4} }{ \left( q^2;q^2 \right)_{2n} 
 \left( q^2;q^2 \right)_{n}^{2} } \frac{1+q^{2n+1}-2q^{4n+1}}{1+q^{2n+1}} 
 = \left[ \begin{matrix} q, q \vspace{1mm} \\ 
 q^2, q^2 \end{matrix} \, \Bigg| \, q^2 \right]_{\infty} $$
 and 
 $$ \sum_{n=0}^{\infty} 
 q^{ n^2} \frac{ \left( q;q^2 \right)_{n}^{2} \left( q^2;q^4 \right)_{n} }{ 
 \left( q^4;q^4 \right)_{n}^{3} } (1-q^{6n+1}) 
 = \left[ \begin{matrix} q^2, q^2 \vspace{1mm} \\ 
 q^4, q^4 \end{matrix} \, \Bigg| \, q^4 \right]_{\infty}, $$
 and 
$$ \sum_{n=0}^{\infty} 
 q^{ 2 n} \frac{ \left( q;q^2 \right)_{n}^{4} }{ 
 \left( q^2;q^2 \right)_{2n} \left( q^2;q^2 \right)_{n}^{2} } \frac{2-q^{2n} - q^{4n+1}}{1+q^{2n+1}} 
 = \left[ \begin{matrix} q, q, q, q \vspace{1mm} \\ 
 q^2, q^2, q^2 \end{matrix} \, \Bigg| \, q^2 \right]_{\infty} \sum_{n=1}^{\infty} \frac{q^{n-1}}{ \left( q;q^2 \right)_{n} }. $$
 The above Guo $q$-analogue of Ramanujan's formula was proved in an alternative way by Chen and Chu 
 through the use of Carlitz inversions of
 the $q$-Pfaff-Saalsch\"{u}tz theorem \cite{ChenChu2023}.

 The first $q$-analogue given in the current subsection
 can be shown to be equivalent to the first $q$-analogue above due to Au \cite{Auunpublished}, 
 via the map $q \mapsto \frac{1}{q}$. 
 By applying the substitution $q \mapsto \frac{1}{q}$, 
 the right-hand side of the first $q$-identity in this subsection 
 becomes the left-hand side of Au's first $q$-identity multiplied by $\frac{1+q}{q}$, 
 and the $q$-Gauss summation formula
 can then be used to evaluate the left-hand side of the first $q$-identity in this subsection. 

\subsection{The Fabry--Guillera formula}\label{subFabryGuillera}
 In Section \ref{subquarter}, we prove the relation 
\begin{multline*}
 \frac{(1-q)^3 (1+q) \left(1 + q^2\right)}{q} \sum_{k = 0}^{\infty} \frac{q^{2k}}{(1-q^{2k+1})^{2} } = \\ \sum_{n=0}^{\infty} q^{2 
 n} \left[ \begin{matrix} 
 q^{2}, q^{2}, q^{2}, q^{4} \vspace{1mm} \\ 
 q^{3}, q^{3} \end{matrix} \, \Bigg| \, q^{2} \right]_{n} 
 \frac{2-q^{2n+1} - q^{4n+3}}{(q^6;q^4)_{n} (q^8;q^4)_{n}} 
\end{multline*}
 for $ q > 1$. By dividing both sides by $2 (1-q)$ and by letting $q \to 1^{+}$, 
 this produces the Ramanujan-type formula 
\begin{equation}\label{FabryGuillera}
 \frac{\pi^2}{4} = 
 \sum_{n = 0}^{\infty} 
 \left( \frac{1}{4} \right)^{n} \left[ \begin{matrix} 
 1, 1, 1 
 \vspace{1mm} \\ 
 \frac{3}{2}, \frac{3}{2}, \frac{3}{2}
 \end{matrix} \right]_{n} (3 n + 2), 
 \end{equation}
 with regard to the survey below, and 
 with \eqref{FabryGuillera} having originally been due (in an equivalent form) to Fabry in 1911
 \cite[p.\ 129]{Fabry1911} 
 and having subsequently been attributed by many authors to Guillera, 
 who independently rediscovered \eqref{FabryGuillera} via the WZ method in 2008 \cite{Guillera2008}. 
 We also introduce a further $q$-analogue of \eqref{FabryGuillera} whereby 
\begin{multline*}
 (1-q) (1-q^2) 
 \sum_{k = 0}^{\infty} 
 \left[ \begin{matrix} 
 q^2 \vspace{1mm} \\ 
 q^{3} \end{matrix} \, \Bigg| \, q^{2} \right]_{k} \frac{q^{2k}}{1-q^{2k+2}} = \\ 
 \sum_{n=0}^{\infty} q^{2n} 
 \frac{ \left( q^2;q^2 \right)_{n}^{4} }{ \left( q^{3};q^{2} \right)_{n}^{2} 
 \left( q^{4};q^{4} \right)_{n} \left( q^{6};q^{4} \right)_{n} }
 \frac{2 - q^{2n+1} - q^{4n+3}}{1 + q^{2n+2}} 
\end{multline*}
 for $ q > 1$. By dividing both sides by $1-q$ and by letting $q \to 1^{+}$, 
 this gives us the Fabry--Guillera relation in \eqref{FabryGuillera}. 
 The first $q$-analogue above is equivalent to a known $q$-analogue reproduced below
 from Hou et al.\ \cite{HouKrattenthalerSun2019}, 
 up to the substitution $q \mapsto \frac{1}{q}$
 and a factor $2 \frac{(1-q^4)(1-q)^2}{q}$.

 There have been a number of past research contributions involving a variety of different methods in the derivation of $q$-analogues 
 of equivalent formulations of the Fabry--Guillera formula in \eqref{FabryGuillera}, 
 which is representative of the research interest in the $q$-analogue shown above
 of the Fabry--Guillera formula. 
 Notably, Hou et al.\ \cite{HouKrattenthalerSun2019} introduced and proved 
 the $q$-analogue
\begin{multline*}
 \frac{1}{2} \sum_{n=0}^{\infty} \frac{q^{2n}}{(1-q^{2n+1})^{2}} = \\ 
 \sum_{n=0}^{\infty} q^{2n(n+1)} 
 \frac{ \left( q^{2};q^2 \right)_{n}^{3} }{ \left( q;q^2 \right)_{n+1}^{3} 
 \left( -1; q \right)_{2n+3} } (1+q^{2n+2} - 2 q^{4n+3})
\end{multline*}
 of the Fabry--Guillera formula, along with the $q$-analogue 
\begin{equation}\label{secondHKS}
 \sum_{n=0}^{\infty} q^{\frac{n(n+1)}{2}} 
 \frac{1-q^{3n+2}}{1-q} \frac{ \left( q;q \right)_{n}^{3} \left( -q;q \right)_{n} }{ 
 \left( q^{3};q^{2} \right)_{n}^{3} } 
 = (1-q)^2 \left[ \begin{matrix} 
 q^{2}, q^{2}, q^{2}, q^{2} \vspace{1mm} \\ 
 q, q, q, q \end{matrix} \, \Bigg| \, q^{2} \right]_{\infty}
\end{equation}
 of the Fabry--Guillera formula. 

 Chen and Chu \cite{ChenChu2021} obtained an equivalent version of \eqref{secondHKS} using the relation 
\begin{multline*}
 \sum_{k=0}^{\infty} 
 \frac{1-q^{3k}a}{1-a} 
 \left[ \begin{matrix} 
 b, d, \frac{q a}{b d} \vspace{1mm} \\ 
 q \end{matrix} \, \Bigg| \, q \right]_{k} 
 \left[ \begin{matrix} 
 a \vspace{1mm} \\ 
 \frac{q^2 a}{b}, \frac{q^2a}{d}, q b d \end{matrix} \, \Bigg| \, q^{2} \right]_{k} q^{\binom{k+1}{2}} 
 = \\ 
 \left[ \begin{matrix} 
 q^2 a, qb, qd, \frac{q^2a}{bd} \vspace{1mm} \\ 
 q, \frac{q^2a}{b}, \frac{q^2 a}{d}, q b d \end{matrix} \, \Bigg| \, q^{2} \right]_{\infty} 
\end{multline*}
 due to Rahman \cite{Rahman1993}. 
 Chu \cite{Chu2018} also obtained an equivalent version of 
 \eqref{secondHKS} 
 using recursions for $q$-analogues of the $\Omega$-sums 
 involved in previous references from Chu et al., 
 as in the work of Chu and Zhang on the acceleration of Dougall's ${}_{5}F_{4}$-summation
 \cite{ChuZhang2014}. 
 Wei \cite{Wei2020} also obtained an equivalent formulation of 
 \eqref{secondHKS} 
 using a terminating $q$-hypergeometric relation due to Chu~\cite{Chu1995} 
 via an inversion technique applied using Jackson's $q$-analogue 
 of the Dougall--Dixon theorem. 
 The variety of published proofs of the $q$-analogue 
 in \eqref{secondHKS} 
 of the Fabry--Guillera formula
 motivate the $q$-analogue of this formula we introduce, 
 via {\tt EKHAD}-normalization. 
 
 Wang and Zhong \cite{WangZhong2023} introduced the $q$-analogue 
\begin{multline*}
 \sum_{n=0}^{\infty} q^{\frac{n^2 + 5n+2}{2}} \frac{ [3n+4]_{q} }{ [2n+3]_{q}^{2} } 
 \frac{ \left( q;q \right)_{n}^{3} \left( -q;q \right)_{n} }{ \left( q^{3};q^{2} \right)_{n}^{3} } 
 = \\ (1-q)^2 \left[ \begin{matrix} 
 q^2, q^2, q^2, q^2 \vspace{1mm} \\ 
 q, q, q, q \end{matrix} \, \Bigg| \, q^{2} \right]_{\infty} - (1+q) 
\end{multline*}
 of an equivalent version of the Fabry--Guillera formula, using a $q$-rational reduction technique. 
 In our previous work \cite{Campbell2025further}, 
 we introduced 
 the $q$-series identity such that 
\begin{multline*} 
 \frac{q^{(k-1)^2}}{(q;q)_{k-1}^2} \sum_{n=0}^{\infty} 
 \frac{q^{n} \left(2 q^k-q^{n+1}-q^{2 n+2}\right)}{1+q^{n+1}} 
 \frac{(q;q)_n^4}{ (q;q)_{2 n+1} (q;q)_{n-k+1}^2} = \\ 
 \frac{1}{(q;q)_{\infty }^4} \sum_{n=0}^{\infty} q^{(k+n)^2} \left((q;q)_{\infty }^3-\left(q^{1-n - 
 k};q\right)_{\infty }^2\right) \left(q^{1+n+k};q\right)_{\infty }^2, 
\end{multline*}
 with the $k = \frac{1}{2}$ case providing a $q$-analogue of the Fabry--Guillera formula. 

\subsection{Ap\'{e}ry's series for $\zeta(2)$}\label{subApery}
 We introduce the $q$-analogue $$ (1-q)^2 (1+q) \sum_{n=0}^{\infty} 
 \frac{q^n}{(1-q^{n+1})^2} = \sum_{n=0}^{\infty} q^{n} 
 \frac{ (q;q)_{n}^{2} }{ (q^3;q^2)_{n} (q^4;q^2)_{n} } 
 (2+q^{n+1}) $$
 for $ q > 1$ 
 of the formula 
\begin{equation}\label{originalApery}
 \frac{\pi^2}{9} = 
 \sum_{k = 0}^{\infty} 
 \left( \frac{1}{4} \right)^{k} \left[ \begin{matrix} 
 1 
 \vspace{1mm} \\ 
 \frac{3}{2} 
 \end{matrix} \right]_{k} \frac{1}{k+1} 
 \end{equation}
 attributed to Ap\'{e}ry, as in the work of Chu and Zhang \cite[Example 74]{ChuZhang2014}. 
 By dividing both sides of the above $q$-identity by 3 and by letting $q \to 1^{+}$, 
 this produces the Ap\'{e}ry formula in \eqref{originalApery}. 

An equivalent version (up to a reindexing) of 
 \eqref{originalApery} played a key step in Ap\'{e}ry's renowned proof \cite{Apery1979,vanderPoorten197879} of the irrationality of 
 $\zeta(2)$ for the Riemann zeta function $\zeta(x) = \sum_{n=1}^{\infty} \frac{1}{n^x}$. This motivates the $q$-analogue of
 Ap\'{e}ry's formula given in this paper, along with number-theoretic applications of this $q$-analogue. Our $q$-analogue of Ap\'{e}ry's 
 series for $\zeta(2)$ is also motivated by Hessami Pilehrood and Hessami Pilehrood's $q$-analogue 
 \cite{HessamiPilehroodHessamiPilehrood2011} of the Bailey--Borwein--Bradley identity, which can be thought of as an extension of 
 Ap\'{e}ry's series for $\zeta(2)$ and $\zeta(3)$. 

\subsection{Zeilberger's series of convergence rate $\frac{1}{64}$}\label{sectionZeilberger}
 A remarkable result due to Zeilberger is 
\begin{equation}\label{displayZeilberger}
 \frac{4\pi^2}{3} = 
 \sum_{k = 0}^{\infty} 
 \left( \frac{1}{64} \right)^{k} \left[ \begin{matrix} 
 1, 1, 1 
 \vspace{1mm} \\ 
 \frac{3}{2}, \frac{3}{2}, \frac{3}{2} 
 \end{matrix} \right]_{k} (21k + 13) 
 \end{equation}
 and was introduced in 1993 \cite{Zeilberger1993} via the WZ method and was subsequently obtained by Guillera \cite{Guillera2008}, 
 again by the WZ method. We introduce the $q$-analogue 
\begin{multline*}
 -(1-q)^3 \sum_{k=0}^{\infty} \frac{q^{k}}{(1-q^{k+1})^2} = \sum_{n=0}^{\infty} q^{n} 
 \frac{ \left( q;q \right)_{n}^{6} }{ \left( q^{2};q^{2} \right)_{n}^{3} 
 \left( q^{3};q^{2} \right)_{n}^{3} } \times \\ 
 \frac{-3 q^{n+1}+q^{2 n+1}-2 q^{2 n+2}+3 q^{3 n+2}+3 q^{4 n+3}+q^{5 n+4}-3}{\left(q^{n+1}+1\right)^3} 
\end{multline*}
 for $ q > 1$ of the Zeilberger formula in \eqref{displayZeilberger}.
 By multiplying both sides by $\frac{8}{q-1}$ and by letting $q \to 1^{+}$, 
 we obtain the Zeilberger formula in \eqref{displayZeilberger}.

 The above $q$-identity may be considered in relation to the previously
 known $q$-analogue 
\begin{multline*}
 (1-q)^{2} \sum_{n=1}^{\infty} \frac{ n q^{n} }{1-q^{n}} = \sum_{n=1}^{\infty} \frac{q^{3n^2-3n+1}}{ \left[ \begin{smallmatrix} 
 2 n \\ 
 n 
 \end{smallmatrix} \right]_{q}^{3} [n]_{q}^{3} } 
 \big( (3q^{4n-1} + 6 q^{3n-1} - \\ q^{2n} + 8 q^{2n-1} - 4 q^{n} + 8 q^{n-1} + 1) [n]_{q} - 8 q^{n} \big) 
\end{multline*}
 of \eqref{displayZeilberger} introduced and proved by 
 Hessami Pilehrood and Hessami Pilehrood in 2011 \cite{HessamiPilehroodHessamiPilehrood2011}. 

 \subsection{A $q$-analogue of a BBP-type formula}
 The $q$-identity 
\begin{multline*}
 q(1 - q^2) 
 \sum_{n=0}^{\infty} q^{2n}
 \left[ \begin{matrix} 
 q \vspace{1mm} \\ 
 q^2 \end{matrix} \, \Bigg| \, q^2 \right]_{n} 
 \frac{1}{1-q^{2n+1}} = \sum_{n=0}^{\infty} (-1)^{n} 
 q^{-n^2-n} \times \\ 
 \frac{ (q;q^2)_{n} (q^2;q^2)_{n} }{(q^4;q^4)_{n} (q^6;q^4)_{n} } 
 \frac{1 - q^{6 n+4}-q^{8 n+3}-q^{8 n+4}+q^{10 n+5}+q^{12 n+6}}{\left(1-q^{4 n+1}\right) \left(1-q^{4 n+3}\right)}. 
\end{multline*}
 for $ q > 1$ is a $q$-analogue of the BBP-type formula $$ \pi = \sum_{n=0}^{\infty} \left( -\frac{1}{4} \right)^{n} \left( \frac{2}{4n+1} + 
 \frac{2}{4n+2} + \frac{1}{4n+3} \right) $$ attributed to Adamchik and Wagon \cite{AdamchikWagon1996}. 
 More specifically, by
 taking $q > 1$ and by letting $q \to 1^{+}$, 
 we obtain the given BBP-type formula.

 The above $q$-identity provides a 
 companion to a $q$-analogue due to Chu (cf.\ \cite{Chu2018}) of the same Adamchik--Wagon BBP-type formula 
 given above.  Apart from the work of Chu 
 \cite{Chu2018}, there does not seem to be much known about $q$-analogues of BBP-type formulas. 

\section{Multiparameter {\tt EKHAD}-normalization}\label{multisec}
 Our $q$-analogues for accelerated hypergeometric series are organized as special cases of the Theorems presented below. 

\subsection{$q$-analogues for the convergence rate $\frac{1}{4}$}\label{subquarter}
 The $q$-WZ pair required to prove the following Theorem is determined by 
 applying {\tt EKHAD}-normalization to a $q$-analogue of a generalization of \eqref{freelowerquarter}, namely 
\begin{equation}\label{prototypeqF}
 F(n, k;q) := q^{k} \left[ \begin{matrix} 
 q^{a}, q^{b} \vspace{1mm} \\ 
 q^{n+c}, q^{n+d} \end{matrix} \, \Bigg| \, q \right]_{k}, 
\end{equation}
 for free parameters $a$, $b$, $c$, and $d$. 

\begin{theorem}\label{theoremquarter}
 Letting

 \ 
 
\noindent $ p(n) = q^{a+b+c+1}-q^{a+b+c+d+n} + q^{a+b+d+1} - q^{a+c+d+n+1}-q^{b+c+d+n+1}+q^{2 c+2 d+3 n}$, 

 \ 

\noindent the relation 
\begin{multline*}
 (q^{a+b} - q^{c+d}) (q^{a+b+1} - q^{c+d}) 
 {}_{3}\phi_{2}\!\!\left[ \begin{matrix} 
 q^{a}, q^{b}, q \vspace{1mm}\\ 
 q^{c}, q^{d} \end{matrix} \ \Bigg| \ q; q 
 \right] = \\ 
 \sum_{n=0}^{\infty} q^{n} 
 \left[ \begin{matrix} 
 q^{-a+c}, q^{-b+c}, q^{-a+d}, q^{-b+d} \vspace{1mm} \\ 
 q^{c}, q^{d} \end{matrix} \, \Bigg| \, q \right]_{n} 
 \frac{p(n)}{ 
 \big( q^{1-a-b+c+d};q^2 \big)_{n} \big( q^{2-a-b+c+d};q^2 \big)_{n} }
 \end{multline*}
 holds provided that $q>1$ and that $a,b,c,d\in\mathbb{C}$ satisfy $\Re(c)>0$, $\Re(d)>0$, and $\Re(c+d-a-b)>1$. 
\end{theorem} 

\begin{proof}
 For convenience, we write $\alpha = q^a$, $\beta = q^b$, $\gamma = q^c$, and $\delta = q^d$. 
 Fix an admissible choice of the parameters, and regard all of the estimates below as pointwise
 estimates for the given fixed choice. Since 
 $q > 1$, the given hypotheses give that 
 $|\gamma| = q^{\Re(c)} > 1$ and that $|\delta| = q^{\Re(d)} > 1$, 
 and we then define 
 $$ r := \left| \frac{\alpha \beta q}{\gamma \delta} \right| = q^{\Re(a + b - c - d+ 1)} < 1.$$
 So, all of the denominator factors occurring in
 $\widetilde{F}(n,k;q)$ and $\widetilde{R}(n,k;q)$ and in the displayed $q$-identity are nonzero.

 Define $$ \widetilde{F}(n, k; q) := 
 q^{k} \left[ \begin{matrix} 
 \alpha, \beta \vspace{1mm} \\ 
 \gamma q^{n}, \delta q^{n} \end{matrix} \, \Bigg| \, q \right]_{k} 
 \left[ \begin{matrix} 
 \frac{\gamma}{\alpha}, \frac{\gamma}{\beta}, \frac{\delta}{\alpha}, \frac{\delta}{\beta} \vspace{1mm} \\ 
 \gamma, \delta \end{matrix} \, \Bigg| \, q \right]_{n} \frac{1}{ \left( \frac{\gamma \delta}{\alpha \beta q}; q \right)_{2n} }, $$ i.e., as the 
 {\tt EKHAD}-normalized version of \eqref{prototypeqF}, according to 
 Lemma \ref{lemmaEKHAD}. 
 We then let 
\begin{multline*}
 \widetilde{R}(n, k; q) = \\ q^{n-k} \frac{ -\alpha \beta \gamma \delta q^{k+n} + \alpha \beta \gamma 
 q + \alpha \beta \delta q - \alpha \gamma \delta q^{n+1} - \beta \gamma \delta q^{n+1} 
 + \gamma^{2} \delta^{2} q^{3n+k} }{(\alpha \beta - \gamma \delta q^{2n}) (\alpha \beta q - \gamma \delta q^{2n})}. 
\end{multline*}
 We then define 
 $\widetilde{G}(n, k; q) := \widetilde{R}(n, k; q) \widetilde{F}(n, k; q)$. We may thus verify that 
 the $q$-WZ difference equation 
\begin{equation}\label{firstdiff}
 \widetilde{F}(n+1, k; q) - \widetilde{F}(n, k; q) = \widetilde{G}(n, k+1; q) - \widetilde{G}(n, k; q) 
\end{equation}
 holds.

 If one of the upper $q$-shifted factorials in $\widetilde{F}(n,k;q)$ terminates, 
 the relevant series is finite and the required convergence is immediate. Thus, for the quotient calculations
 below, we may assume that no such termination occurs. 
 
 We now verify the hypotheses of Lemma \ref{lemmaRectangularTelescope}.
 For each fixed $k \in \mathbb{N}_{0}$, we have 
 $$ \lim_{n \to \infty} \left| \frac{\widetilde{F}(n+1,k;q)}{\widetilde{F}(n,k;q)} \right| = 0, $$
 so that 
 $$ \lim_{n \to \infty} \widetilde{F}(n,k;q) = 0. $$
 The ratio test then gives us that 
 $$ \sum_{n=0}^{\infty} \left| \widetilde{F}(n,k;q) \right| < \infty. $$

 We now make use of the explicit estimate
 $$ \left| \widetilde{R}(n,k;q) \right| 
 \leq \frac{|\alpha \beta \gamma \delta| + q |\alpha \beta \gamma| 
 + q | \alpha \beta \delta | + q |\alpha \gamma \delta| 
 + q |\beta \gamma \delta| + |\gamma^2 \delta^2| }{ 
 |\gamma\delta|^2 (1-r)^2 } =: M.$$
 This gives us that 
 $$ \left| \widetilde{G}(n,k;q) \right| \leq M \left| \widetilde{F}(n,k;q) \right|, $$
 so that 
 $$ \sum_{n=0}^{\infty} \left| \widetilde{G}(n,k;q) \right| < \infty $$
 for each fixed $k \in \mathbb{N}_{0}$. 
 Also, we have that 
 $$ \lim_{k \to \infty} \left| \frac{\widetilde{F}(0,k+1;q)}{\widetilde{F}(0,k;q)} \right| 
 = \left| \frac{\alpha \beta q}{\gamma \delta} \right| < 1, $$
 giving us that 
 $$ \sum_{k=0}^{\infty} \left| \widetilde{F}(0,k;q) \right| < \infty.$$
 So, it remains to verify the uniform geometric bound in the final hypothesis
 of Lemma \ref{lemmaRectangularTelescope}. 
 
 We proceed to select
 a value $\rho$ satisfying $r < \rho < 1$. 
 For elements $n$ and $k$ in $\mathbb{N}_{0}$, define
\begin{equation}\label{definecalA}
 \mathcal{A}_{n, k} = q^{-2n} \frac{(1 - \alpha^{-1} q^{-k}) (1 - \beta^{-1} q^{-k}) }{(1 - \gamma^{-1} q^{-n-k}) (1 - 
 \delta^{-1} q^{-n-k})}. 
\end{equation}
 From the definition in \eqref{definecalA}, we find that 
 $$ \left| \frac{\widetilde{F}(n,k+1;q)}{\widetilde{F}(n,k;q)} \right| = r \, \left| \mathcal{A}_{n,k} \right|. $$
 For all sufficiently large $k$, uniformly for $n \geq 0$, the relation
\begin{equation}\label{calAbound} 
 \left| \mathcal{A}_{n, k} \right| \leq \frac{ (1 + |\alpha|^{-1} |q|^{-k}) (1 + |\beta|^{-1} |q|^{-k}) 
 }{ (1 - |\gamma|^{-1} |q|^{-k}) (1 - |\delta|^{-1} |q|^{-k}) } 
\end{equation}
  holds.  The right-hand side of \eqref{calAbound} approaches $1$ as $k \to \infty$.
 Consequently, there is an element $k_0$ of $\mathbb{N}_{0}$, 
 independent of $n$, satisfying 
 $$ \left| \frac{ \widetilde{F}(n,k+1;q) }{\widetilde{F}(n,k;q)} \right| \leq \rho, $$
 for each $n \geq 0$ and each $k \geq k_0$. 
 Therefore, the relation 
 $$ \left| \widetilde{F}(n,k;q) \right| \leq \rho^{k- k_0} \left| 
 \widetilde{F}(n,k_0;q) \right| $$
 holds,  and thus
 $$ \left| \widetilde{G}(n,k;q) \right| 
 \leq M \, \rho^{k-k_0} \left| \widetilde{F}(n,k_0;q) \right|. $$
 Since 
 $$ \sum_{n=0}^{\infty} \left| \widetilde{F}(n,k_0;q) \right| < \infty, $$
 the final hypothesis of Lemma \ref{lemmaRectangularTelescope} holds, with 
 $ A_n = \left| \widetilde{F}(n,k_0;q) \right|$. 

 So, since the hypotheses of Lemma \ref{lemmaRectangularTelescope} are satisfied, 
 we have that 
 $$\sum_{k=0}^{\infty} \widetilde{F}(0,k;q) = \sum_{n=0}^{\infty}\widetilde{G}(n,0;q),$$ 
 for parameters satisfying the given conditions. 
 Indeed, the summand $\widetilde{F}(0$, $k$; $q)$ is the $k^{\text{th}}$ summand of the displayed ${}_{3}\phi_2$-series, while 
 $(q^{a+b} - q^{c+d})(q^{a+b+1}-q^{c+d}) \widetilde{G}(n,0;q)$
 simplifies to the $n^{\text{th}}$ summand on the right-hand side of Theorem \ref{theoremquarter}. 
 \end{proof}

 Theorem \ref{theoremquarter} is a multiparameter $q$-transformation. In each specialization below, the phrase ``$q$-analogue'' refers 
 to the explicitly stated normalized limit as $q \to 1^{+}$, 
 and similarly for the remaining Theorems introduced in this paper. 

\begin{example}
 Setting $(a, b, c, d) = \big( \frac{1}{2}, \frac{1}{2}, 2, 2 \big)$ in Theorem \ref{theoremquarter}, 
 we obtain the $q$-identity 
 highlighted as a motivating result in Section \ref{subRamanujan}. 
\end{example}

\begin{example}
 Setting $(a, b, c, d) = \big( \frac{1}{2}, \frac{1}{2}, \frac{3}{2}, \frac{3}{2} \big)$ 
 in Theorem \ref{theoremquarter}, we obtain 
 the first $q$-identity highlighted as a motivating result in Section \ref{subFabryGuillera}. 
\end{example}

\begin{example}
 Setting $(a, b, c, d) = \big( 1, 1, 2, 2 \big)$ in Theorem \ref{theoremquarter}, we obtain 
 the $q$-identity highlighted as a motivating result in Section \ref{subApery}. 
\end{example}
 
\begin{example}
 Setting $(a, b, c, d) = 
 \big( \frac{1}{2}, 1, \frac{3}{2}, \frac{3}{2} \big)$ in Theorem \ref{theoremquarter}, we obtain the $q$-identity 
\begin{multline*}
 (1-q^3) \sum_{n=0}^{\infty} 
 \left[ \begin{matrix} 
 q^2 \vspace{1mm} \\ 
 q^3 \end{matrix} \, \Bigg| \, q^2 \right]_{n} \frac{q^{2n}}{1-q^{2n+1}} = \\ 
 \sum_{n=0}^{\infty} q^{2n} 
 \frac{(q^2;q^2)_{n}^{2}}{(q^5;q^4)_{n} (q^7;q^4)_{n}} \frac{2 - q^{2n+2} - q^{4n+3}}{1-q^{2n+1}}. 
 \end{multline*}
 By letting $q > 1$ and by letting $q \to 1^{+}$, 
 we obtain the formula 
\begin{equation*}
 6 G = 
 \sum_{n = 0}^{\infty} 
 \left( \frac{1}{4} \right)^{n} \left[ \begin{matrix} 
 1, 1 
 \vspace{1mm} \\ 
 \frac{5}{4}, \frac{7}{4} 
 \end{matrix} \right]_{n} \frac{6n+5}{2n+1} 
 \end{equation*}
 proved by Chu and Zhang \cite[Example 84]{ChuZhang2014}. 
\end{example}

\begin{example}
 Setting $(a, b, c, d) = \big( -\frac{1}{2}, \frac{1}{6}, \frac{1}{3}, 1 \big)$ in Theorem \ref{theoremquarter}, 
 and then applying the substitution $ q \to q^6$ yields
\begin{multline*}
 \left(\frac{1}{q^2}-q^8\right) \left(q^4-q^8\right) 
 \sum_{k=0}^{\infty} q^{6k} 
 \left[ \begin{matrix} 
 q^{-3}, q \vspace{1mm} \\ 
 q^2, q^6 \end{matrix} \, \Bigg| \, q^6 \right]_{k} 
 = \\ 
 \sum_{n=0}^{\infty} q^{6n+6} 
 \left[ \begin{matrix} 
 q, q^5, q^5, q^9 \vspace{1mm} \\ 
 q^2, q^6 \end{matrix} \, \Bigg| \, q^6 \right]_{n} 
 \frac{ \left(1-q^{6 n+5}\right) \left(-q^{6 n}-q^{12 n+5}+q^4+1\right) }{ 
 \big( q^{16};q^{12} \big)_{n} \big( q^{22};q^{12} \big)_{n} }
 \end{multline*}
 and then letting $q > 1$, and by then multiplying both sides of the resulting $q$-identity 
 by $\frac{36}{5(1 - q^6)^2}$, and by then letting $q \to 1^{+}$, 
 this yields the Ramanujan-type formula
\begin{equation*}
 \frac{8\sqrt{3}}{3} = 
 \sum_{n = 0}^{\infty} 
 \left( \frac{1}{4} \right)^{n} \left[ \begin{matrix} 
 \frac{1}{6}, \frac{5}{6}, \frac{3}{2} 
 \vspace{1mm} \\ 
 \frac{1}{3}, 1, \frac{4}{3} 
 \end{matrix} \right]_{n} (18n + 1). 
 \end{equation*}
\end{example}

\begin{remark}
For each of the examples below, the corresponding $q$-identity is obtained
explicitly by substituting the displayed values of $(a,b,c,d)$ into the
indicated master theorem. To avoid repeatedly reproducing lengthy
specializations of the same theorem, we record the specialized $q$-identity
in full when comparison with a previously published evaluation is relevant;
otherwise, we state the parameter specialization, the required normalizing
factor, and the resulting classical limit.
\end{remark}

\begin{example}
 Setting $(a, b, c, d) = \big( \frac{1}{3}, \frac{1}{3}, 1, 1 \big)$ in Theorem \ref{theoremquarter}, 
 and then letting $q > 1$, by multiplying both sides of the resulting $q$-identity 
 by $\frac{9}{2(1-q)^2}$, 
 and by then letting $q \to 1^{+}$, 
 this produces the Ramanujan-type formula 
\begin{equation*}
 \frac{3 \Gamma^{3}\left( \frac{1}{3} \right) }{2 \pi^2} = 
 \sum_{n = 0}^{\infty} 
 \left( \frac{1}{4} \right)^{n} \left[ \begin{matrix} 
 \frac{2}{3}, \frac{2}{3}, \frac{2}{3} 
 \vspace{1mm} \\ 
 1, 1, \frac{7}{6} 
 \end{matrix} \right]_{n} (9 n + 2). 
 \end{equation*}
\end{example}

\begin{example}
 Setting $(a, b, c, d) = 
 \big( \frac{1}{6}, \frac{1}{6}, 1, 1 \big)$ in Theorem \ref{theoremquarter}, and by then letting $q > 1$, 
 and by then multiplying both sides of the resulting $q$-identity by $\frac{36}{5 (1-q)^2}$, 
 and by then letting $q \to 1^{+}$, we obtain 
 the Ramanujan-type formula 
\begin{equation*}
 \frac{2^{4/3} \sqrt{3} \Gamma^{3}\left( \frac{1}{3} \right)}{\pi^2} = 
 \sum_{n = 0}^{\infty} 
 \left( \frac{1}{4} \right)^{n} \left[ \begin{matrix} 
 \frac{5}{6}, \frac{5}{6}, \frac{5}{6} 
 \vspace{1mm} \\ 
 1, 1, \frac{4}{3} 
 \end{matrix} \right]_{n} (18n+5). 
 \end{equation*}
\end{example}

\begin{example}
 Setting $(a, b, c, d) = \big( \frac{1}{3}, -\frac{1}{2}, 1, \frac{1}{6} \big)$ in Theorem \ref{theoremquarter},
 and then setting $q > 1$, 
 and then multiplying both sides of the resulting $q$-identity by 
 $\frac{9}{(1-q)^2}$, 
 and by then letting $q \to 1^{+}$, 
 this gives us the Ramanujan-type formula 
\begin{equation*}
 -\frac{ \Gamma^{3}\left( \frac{1}{6} \right) }{3 \, 2^{2/3} \pi^{3/2} } = 
 \sum_{n = 0}^{\infty} 
 \left( \frac{1}{4} \right)^{n} \left[ \begin{matrix} 
 -\frac{1}{6}, \frac{2}{3}, \frac{3}{2} 
 \vspace{1mm} \\ 
 \frac{1}{6}, 1, \frac{7}{6} 
 \end{matrix} \right]_{n} (18 n - 1). 
 \end{equation*}
\end{example}

\subsection{$q$-analogues for the convergence rate $-\frac{1}{4}$}
 By analogy with the input function in \eqref{prototypeqF}, we set 
\begin{equation}\label{inputnegquart}
 F(n, k;q) := q^{k} \left[ \begin{matrix} 
 q^{a}, q^{b+n} \vspace{1mm} \\ 
 q^{n+c}, q^{2n+d} \end{matrix} \, \Bigg| \, q \right]_{k}, 
\end{equation}
 and apply {\tt EKHAD}-normalization to \eqref{inputnegquart}, and, by mimicking the proof of Theorem \ref{theoremquarter}, this 
 provides a powerful method for generating
 $q$-analogues of accelerated series of convergence rate $-\frac{1}{4}$. 

\begin{theorem}\label{theoremnegquart}
 For the $q$-polynomial

 \ 

\noindent $p(n) = q^{n (a-b-1)} \big(-q^{a+b+c+d+3 n}-q^{a+b+c+d+3 n+1}+q^{a+b+c+2 d+5 n}-q^{a+b+2 d+4 n+1}+q^{2 a+b+1}+q^{a+c+2 d+4 n+1}-q^{2 a+d+n+1}+q^{a+2
 d+3 n+1}+q^{b+c+2 d+5 n+1}-q^{c+2 d+4 n+1}+q^{2 c+2 d+5 n}-q^{2 c+3 d+7 n}\big)$, 

 \ 

\noindent the relation 
\begin{multline*}
 (1-q^d) (q^{c+d} - q^{a+b}) (q^{c+d} - q^{a+b+1}) 
 {}_{3}\phi_{2}\!\!\left[ \begin{matrix} 
 q^{a}, q^{b}, q \vspace{1mm}\\ 
 q^{c}, q^{d} \end{matrix} \ \Bigg| \ q; q 
 \right] \\ 
 = \sum_{n=0}^{\infty} (-1)^{n} q^{-\binom{n}{2}} 
 \left[ \begin{matrix} 
 q^{-a+d}, q^{-a+d+1} \vspace{1mm} \\ 
 q^{d+1}, q^{d+2}, q^{-a-b+c+d+1}, q^{-a-b+c+d+2} \end{matrix} \, \Bigg| \, q^2 \right]_{n} \times \\ 
 \left[ \begin{matrix} 
 q^{b}, q^{-a+c}, q^{-b+d} \vspace{1mm} \\ 
 q^{c} \end{matrix} \, \Bigg| \, q \right]_{n} \, p(n)
\end{multline*}
 holds provided that $q>1$ and that $a,b,c,d\in\mathbb{C}$ satisfy $\Re(c)>0$, $\Re(d)>0$, and $\Re(c+d-a-b)>1$. 
\end{theorem} 

\begin{proof}
 Applying the $q$-version of Zeilberger's algorithm to \eqref{inputnegquart}, 
 this produces the rational certificate

 \ 

\noindent $R(n, k;q) = \frac{1}{q^{d+k+2 n}-1} q^{1-k} \big(q^{c+n}-1\big) \big(q^{d+2 n}-1\big) \big(q^{d+2 n+1}-1\big) \big(q^{a+b+c+d+k+3 n}+q^{a+b+c+d+k+3 n+1}-q^{a+b+c+2 d+2 k+5
 n}+q^{a+b+2 d+k+4 n+1}-q^{2 a+b+1}-q^{a+c+2 d+k+4 n+1}+q^{2 a+d+n+1}-q^{a+2 d+3 n+1}-q^{b+c+2 d+k+5 n+1}-q^{2 c+2 d+k+5 n}+q^{2 c+3 d+2 k+7
 n} + $

\noindent $q^{c+2 d+4 n+1}\big)$,

 \ 

\noindent writing $H(n,k;q) = R(n,k;q) F(n,k;q)$, together with the polynomial coefficients

 \ 

\noindent $p_{1}(n;q) = q^2 \big(q^{b+n}-1\big) \big(q^a-q^{c+n}\big) \big(q^a-q^{d+2 n}\big) \big(q^{d+2 n+1}-q^a\big) \big(q^b-q^{d+n}\big)$

 \ 

\noindent and 

 \ 

\noindent $p_{2}(n;q) = q^{n+1} \big(q^{c+n}-1\big) \big(q^{d+2 n}-1\big) \big(q^{d+2 n+1}-1\big) \big(q^{a+b+1}-q^{c+d+2 n}\big) \big(q^{c+d+2
 n}-q^{a+b}\big)$, 

 \ 

\noindent so that the first-order $q$-difference equation in 
 \eqref{generalqZeil} holds. 
 An application of {\tt EKHAD}-normalization, via Lemma \ref{lemmaEKHAD}, 
 then produces the normalized companion
\begin{multline*}
 \widetilde{F}(n,k;q)
 = (-1)^n q^{\frac{1}{2} n (2 a-2 b-n+1)+k} \times \\ 
 \left[ \begin{matrix} 
 q^{-a+d}, q^{-a+d+1} \vspace{1mm} \\ 
 q^d, q^{d+1}, q^{-a-b+c+d-1}, q^{-a-b+c+d} \end{matrix} \, \Bigg| \, q^2 \right]_{n} \times \\ 
 \left[ \begin{matrix} 
 q^b, q^{-a+c}, q^{-b+d} \vspace{1mm} \\ 
 q^c \end{matrix} \, \Bigg| \, q \right]_{n} 
 \left[ \begin{matrix} 
 q^{b+n}, q^a \vspace{1mm} \\ 
 q^{d+2 n}, q^{c+n} \end{matrix} \, \Bigg| \, q \right]_{k} 
\end{multline*}
 to $F(n, k; q)$. 
 The rational certificate associated with 
 $ \widetilde{F}(n,k;q)$ is

 \ 

\noindent $\widetilde{R}(n,k;q) 
 = \frac{q^{-k-n}}{\left(1-q^{d+k+2 n}\right) \left(q^{a+b+1}-q^{c+d+2 n}\right) \left(q^{c+d+2 n}-q^{a+b}\right)} 
 \big( q^{a+b+c+d+k+3 n} + $

\noindent $ q^{a+b+c+d+k+3 n+1}-q^{a+b+c+2 d+2 k+5 n}+q^{a+b+2 d+k+4 n+1}-q^{2 a+b+1}-q^{a+c+2 d+k+4 n+1}+q^{2 a+d+n+1}-q^{a+2 d+3
 n+1}-q^{b+c+2 d+k+5 n+1}-q^{2 c+2 d+k+5 n}+q^{2 c+3 d+2 k+7 n} + $

\noindent $q^{c+2 d+4 n+1} \big)$,

 \ 

\noindent so that, by setting $\widetilde{G}(n,k;q) = \widetilde{R}(n,k;q) \widetilde{F}(n,k;q)$, 
 we have that the $q$-WZ difference equation in \eqref{qWZequation} holds true.

If a numerator $q$-shifted factorial terminates, then the relevant series
 is finite and the required convergence is immediate. Thus, the following ratio-test argument concerns
 only the nonterminating case. 
In the nonterminating case, direct ratio estimates from the displayed formulas for
$\widetilde{F}$ and $\widetilde{G}=\widetilde{R}\widetilde{F}$ show, for every fixed
admissible $(a,b,c,d)$, that $\widetilde{F}(N,k;q)\to0$ as $N\to\infty$ for each
fixed $k$, that $\sum_{k=0}^{\infty}|\widetilde{F}(0,k;q)|$ converges and
$\sum_{n=0}^{\infty}|\widetilde{G}(n,k;q)|$ converges for each fixed $k$, and that
there exist $k_{0}\geq0$, $M>0$, $\rho\in(0,1)$, and a nonnegative sequence
$(A_n)_{n\geq0}$ with $\sum_{n=0}^{\infty}A_n<\infty$ such that
$|\widetilde{G}(n,k;q)|\leq M\rho^{\,k-k_{0}}A_n$ for all $n\geq0$ and
$k\geq k_{0}$; hence all the hypotheses of
Lemma~\ref{lemmaRectangularTelescope} are satisfied.

Therefore, Lemma \ref{lemmaRectangularTelescope} yields the stated $q$-series identity. 
 Moreover, the 
 expression
 $\widetilde{F}(0,k;q)$ is the $k^{\text{th}}$ summand of the displayed
 ${}_{3}\phi_{2}$-series, while 
 $(1-q^d)(q^{c+d} - q^{a+b})(q^{c+d}-q^{a+b+1}) \widetilde{G}(n,0;q)$
 simplifies to the $n^{\text{th}}$ summand on the right-hand side of Theorem \ref{theoremnegquart}.
 \end{proof}

\begin{example}
 Setting $(a, b, c, d) = \big( \frac{1}{2}, 1, \frac{3}{2}, \frac{3}{2} \big)$ in Theorem \ref{theoremnegquart}, we obtain 
 the $q$-identity 
\begin{multline*}
 q^2 (1-q) (1-q^3)^{2} 
 \sum_{n=0}^{\infty} 
 \left[ \begin{matrix} 
 q^2 \vspace{1mm} \\ 
 q^3 \end{matrix} \, \Bigg| \, q^2 \right]_{n} 
 \frac{q^{2n}}{1 - q^{2n+1}} = 
 \sum_{n=0}^{\infty} (-1)^{n} 
 q^{-n^2-2n} \times \\ 
 \left[ \begin{matrix} 
 q^2, q^4 \vspace{1mm} \\ 
 q^5, q^5, q^7, q^7 \end{matrix} \, \Bigg| \, q^4 \right]_{n} (q^2;q^2)_{n}^{2} \, 
 \big( 1 - q^{6 n+5} - 2 q^{8 n+5}+q^{10 n+7}+q^{12 n+8} \big). 
\end{multline*}
 By then letting $q > 1$, and by then dividing both sides of this $q$-identity by $(1-q)^2$, 
 and by then letting $q \to 1^{+}$, this produces the formula 
\begin{equation*}
 18 G = 
 \sum_{n = 0}^{\infty} 
 \left( -\frac{1}{4} \right)^{n} \left[ \begin{matrix} 
 \frac{1}{2}, 1, 1, 1 
 \vspace{1mm} \\ 
 \frac{5}{4}, \frac{5}{4}, \frac{7}{4}, \frac{7}{4} 
 \end{matrix} \right]_{n} \big( 40	 n^2 + 56n + 19 \big) 
 \end{equation*}
 due to Chu and Zhang \cite[Example 85]{ChuZhang2014}. 
\end{example}

\begin{example}
 Setting $(a, b, c, d) = \big( 1, 1, 2, 2 \big)$ in Theorem \ref{theoremnegquart}, we obtain 
 the $q$-identity 
\begin{multline*}
 q (1-q^2) 
 \sum_{n=0}^{\infty} 
 \frac{q^n}{(1-q^{n+1})^2} = \\ \sum_{n=0}^{\infty} (-1)^{n} 
 q^{-\binom{n}{2}-n} 
 \frac{ (q;q)_{n} (q^2;q)_{n} }{ (q^3;q^2)_{n} (q^4;q^2)_{n} }
 \frac{1 + q^{n+1} + q^{2n+2} - 2 q^{4n+3} - q^{5n+4}}{ (1-q^{2n+1}) (1-q^{2n+2}) }, 
\end{multline*}
 and, then, by letting $q > 1$, and by then multiplying both sides of the above $q$-identity by $2(1-q)$
 and by then letting $q \to 1^{+}$, 
 this produces the formula 
\begin{equation}\label{20251q129wwww411qqqqqqqqqqqqq103AM1A}
 \frac{2\pi^2}{3} = 
 \sum_{n = 0}^{\infty} 
 \left( -\frac{1}{4} \right)^{n} \left[ \begin{matrix} 
 1 
 \vspace{1mm} \\ 
 \frac{3}{2} \end{matrix} \right]_{n} \frac{10 n + 7}{(n+1)(2n+1)}, 
\end{equation}
 proved by Chu via well poised $\Omega$-sums \cite[Example 33]{Chu2021Omega}. 
\end{example}

\begin{example}
 Setting $(a, b, c, d) = \big( 1, 1, \frac{3}{2}, 2 \big)$ in Theorem \ref{theoremnegquart}, we obtain the $q$-identity 
\begin{multline*}
 q^2 (1-q^3) 
 \sum_{n=0}^{\infty} 
 \left[ \begin{matrix} 
 q^2 \vspace{1mm} \\ 
 q^3 \end{matrix} \, \Bigg| \, q^2 \right]_{n} 
 \frac{q^{2n}}{1 - q^{2n+2}} = \sum_{n=0}^{\infty} (-1)^{n} 
 q^{-n^2-n} 
 \frac{(q^2;q^2)_{n}^{2}}{(q^5;q^4)_{n}(q^7;q^4)_{n} } \times \\ 
 \frac{1 + q^{2 n+1}+q^{4 n+2}+q^{6 n+4}-q^{6 n+5}-q^{8 n+5}-q^{8 n+6}-q^{10 n + 
 7} }{\left(1-q^{2 n+1}\right) \left(1 + q^{2 n+1}\right) \left(1 + q^{2
 n+2}\right)}. 
\end{multline*}
 By then letting $q > 1$ and letting $q \to 1^{+}$, we obtain the formula
\begin{equation}\label{20q25q1qq12q9wrwrrwwr4r1e1e1e0e3wAwqM1qA}
 \frac{3\pi^2}{8} = 
 \sum_{n = 0}^{\infty} 
 \left( -\frac{1}{4} \right)^{n} \left[ \begin{matrix} 
 1, 1 
 \vspace{1mm} \\ 
 \frac{5}{4}, \frac{7}{4} 
 \end{matrix} \right]_{n} \frac{5n+4}{2n+1}, 
 \end{equation}
 proved by Chu via well poised $\Omega$-sums \cite[Example 32]{Chu2021Omega}. 
 Setting $(a, b, c, d) = \big( \frac{1}{2}, \frac{1}{2}, \frac{3}{2}, \frac{3}{2} \big)$, 
 we obtain the 
 $q$-identity 
\begin{multline*}
 (1-q)^2 q (1+q+q^2) 
 \sum_{n=0}^{\infty} 
 \frac{q^{2n}}{(1-q^{2n+1})^{2} } = 
 \sum_{n=0}^{\infty} (-1)^{n} 
 q^{-n(n+1)} 
 \frac{ \left( q^2; q^2 \right)_{n}^{2} }{ \left( q^{5};q^{4} \right)_{n} 
 \left( q^{7};q^{4} \right)_{n} } \times \\ 
 \frac{-q^{2 n+1}-q^{4 n+2}-q^{6 n+4}+q^{6 n+5}+q^{8 n+5}+q^{8 n+6}+q^{10 n+7}-1}{\left(q^{2 n+1}-1\right) \left(q^{2 n+1}+1\right) \left(q^{2
 n+2}+1\right)}, 
\end{multline*}
 and, by letting $q > 1$ and $q \to 1^{+}$, we again obtain the Chu formula in 
 \eqref{20q25q1qq12q9wrwrrwwr4r1e1e1e0e3wAwqM1qA}. 
\end{example}

\begin{example}
 Setting $(a, b, c, d) = \big( \frac{1}{6}, \frac{1}{6}, \frac{2}{3}, 1 \big)$ in Theorem \ref{theoremnegquart}, for the $q$-polynomial

 \ 

\noindent $ p(n) = \big(1-q^{2 n+1}\big) 
 \big(1 + q^{2 n+1}+q^{4 n+2}\big) 
 \big(1 + q^{6 n+1}+q^{6 n+3}-q^{6 n+5}+q^{12 n+2}+q^{12 n+4}-q^{12 n+8}+q^{18
 n+5}-q^{18 n+9}+q^{18 n+10}-q^{18 n+11}-q^{24 n+10}-q^{24 n+12}-q^{30 n+13}\big), $

 \ 

\noindent we obtain the $q$-identity 
\begin{multline*}
 q ( 1 + q) \left(1-q^6\right) \left(1-q^8\right) 
 \sum_{n=0}^{\infty} q^{6n} 
 \left[ \begin{matrix} 
 q, q \vspace{1mm} \\ 
 q^4, q^6 \end{matrix} \, \Bigg| \, q^6 \right]_{n} 
 = \\ \sum_{n=0}^{\infty} (-1)^{n} q^{-3n^2-3n} 
 \left[ \begin{matrix} 
 q^5, q^{11} \vspace{1mm} \\ 
 q^{12}, q^{14}, q^{18}, q^{20} \end{matrix} \, \Bigg| \, q^{12} \right]_{n} 
 \left[ \begin{matrix} 
 q^{3}, q^{5}, q^{7} \vspace{1mm} \\ 
 q^{4} \end{matrix} \, \Bigg| \, q^6 \right]_{n} \, p(n). 
\end{multline*}
 By then letting $q > 1$, and by then dividing both sides of the displayed $q$-identity
 by $3 (1-q)^2$, and by then letting $q \to 1^{+}$, 
 this produces the Ramanujan-type formula 
\begin{equation*}
 \frac{64\sqrt{3}}{3} = 
 \sum_{n = 0}^{\infty} 
 \left( -\frac{1}{4} \right)^{n} \left[ \begin{matrix} 
 \frac{5}{12}, \frac{5}{6}, \frac{11}{12} 
 \vspace{1mm} \\ 
 \frac{2}{3}, 1, \frac{5}{3} 
 \end{matrix} \right]_{n} (60 n + 43) 
 \end{equation*}
 introduced and proved by Chu in 2021 \cite{Chu2021Omega}
 via well poised $\Omega$-sums. 
\end{example}

\begin{example}
 Setting $(a, b, c, d) = \big( \frac{1}{3}, \frac{1}{6}, 1, \frac{5}{6} \big)$ in Theorem \ref{theoremnegquart},
 and by then letting $q > 1$, 
 and by then multiplying both sides of the resulting $q$-identity 
 by $\frac{18}{(1-q)^3}$, and by then letting $q \to 1^{+}$, 
 this produces the Ramanujan-type formula 
\begin{equation*}
 \frac{20 \sqrt[3]{2} }{3} = 
 \sum_{n = 0}^{\infty} 
 \left( -\frac{1}{4} \right)^{n} \left[ \begin{matrix} 
 \frac{1}{4}, \frac{2}{3}, \frac{3}{4} 
 \vspace{1mm} \\ 
 \frac{11}{12}, 1, \frac{17}{12}
 \end{matrix} \right]_{n} (20 n + 9) 
 \end{equation*}
 established previously via a different method
 in our past work \cite{Campbellunpublished}. 
\end{example}

\begin{example}
 Setting $(a, b, c, d) = \big( \frac{1}{2}, 1, \frac{3}{2}, 2 \big)$ in Theorem \ref{theoremnegquart},
 and by then setting $q > 1$, and by then multiplying both sides of the resultant $q$-identity
 by $\frac{2}{(1-q)^{3}}$, and by then setting $q \to 1^{+}$, 
 this produces the 
 Ramanujan-type formula 
\begin{equation}\label{20251q129411ee1e0e3eqAqMq1qA}
 16 \log 2 = 
 \sum_{n = 0}^{\infty} 
 \left( -\frac{1}{4} \right)^{n} \left[ \begin{matrix} 
 \frac{3}{4}, 1, \frac{5}{4} 
 \vspace{1mm} \\ 
 \frac{3}{2}, \frac{3}{2}, \frac{3}{2}
 \end{matrix} \right]_{n} (20 n + 13) 
 \end{equation}
 proved via a different method in our past work \cite{Campbellunpublished}. 
\end{example}

 \begin{example}
 Setting $(a, b, c, d) = \big( \frac{3}{4}, \frac{1}{4}, \frac{3}{2}, 1 \big)$ in Theorem \ref{theoremnegquart}, 
 and by then letting $q > 1$, and by then multiplying both sides of the resulting 
 $q$-identity by $\frac{64}{3 (1-q)^3}$, 
 and by then letting $q \to 1^{+}$, 
 this produces the Ramanujan-type formula 
\begin{equation*}
 16 \sqrt{2} = 
 \sum_{n = 0}^{\infty} 
 \left( -\frac{1}{4} \right)^{n} \left[ \begin{matrix} 
 \frac{1}{8}, \frac{5}{8}, \frac{3}{4} 
 \vspace{1mm} \\ 
 1, \frac{3}{2}, \frac{3}{2} 
 \end{matrix} \right]_{n} (40 n + 23) 
 \end{equation*}
 proved via a different method in our past work \cite{Campbellunpublished}. 
\end{example}

 \begin{example}
 Setting $(a, b, c, d) = \big( \frac{3}{4}, 1, \frac{3}{2}, \frac{3}{2} \big)$ in Theorem \ref{theoremnegquart}, 
 and then letting $q > 1$, by then multiplying both sides of the resulting $q$-identity
 by $\frac{32}{3 (1-q)^3}$, and by then letting $q \to 1^{+}$, 
 we obtain the Ramanujan-type formula 
\begin{equation*}
 \frac{5 \Gamma^{4}\left( \frac{1}{4} \right) }{16 \pi} = 
 \sum_{n = 0}^{\infty} 
 \left( -\frac{1}{4} \right)^{n} \left[ \begin{matrix} 
 \frac{3}{8}, \frac{7}{8}, 1 
 \vspace{1mm} \\ 
 \frac{9}{8}, \frac{5}{4}, \frac{13}{8} 
 \end{matrix} \right]_{n} (40 n + 19). 
 \end{equation*}
\end{example}

\begin{example}
 Setting $(a, b, c, d) = \big( \frac{1}{4}, \frac{1}{4}, 1, 1 \big)$ in Theorem \ref{theoremnegquart}, 
 and by then letting $q > 1$, 
 and by multiplying both sides of the resulting $q$-identity by $\frac{64}{3(1-q)^3}$, 
 and by then letting $q \to 1^{+}$, 
 this produces the Ramanujan-type formula 
\begin{equation*}
 \frac{8 \Gamma^{2}\left( \frac{1}{4} \right) }{\pi^{3/2}} = 
 \sum_{n = 0}^{\infty} 
 \left( -\frac{1}{4} \right)^{n} \left[ \begin{matrix} 
 \frac{3}{8}, \frac{3}{4}, \frac{7}{8} 
 \vspace{1mm} \\ 
 1, 1, \frac{3}{2}
 \end{matrix} \right]_{n} (40 n + 21). 
 \end{equation*}
\end{example}

 \begin{example}
 Setting $(a, b, c, d) = \big( \frac{1}{2}, \frac{5}{6}, 1, 2 \big)$ in Theorem \ref{theoremnegquart}, 
 and then letting $q > 1$, and by then multiplying both sides of the resulting $q$-identity by 
 $\frac{24}{5(1-q)^3}$, 
 and by then letting $q \to 1^{+}$, 
 this produces the Ramanujan-type formula 
\begin{equation*}
 \frac{2^{\frac{25}{3}} \pi}{3 \Gamma^{3}\left( \frac{1}{3} \right) } = 
 \sum_{n = 0}^{\infty} 
 \left( -\frac{1}{4} \right)^{n} \left[ \begin{matrix} 
 \frac{3}{4}, \frac{7}{6}, \frac{5}{4} 
 \vspace{1mm} \\ 
 1, \frac{4}{3}, 2 
 \end{matrix} \right]_{n} (20 n + 21). 
 \end{equation*}
\end{example}

\begin{example}
 Setting $(a, b, c, d) = \big( \frac{3}{4}, \frac{1}{2}, 1, \frac{3}{2} \big)$ in Theorem \ref{theoremnegquart}, and by 
 then setting $q > 1$, 
 and by then multiplying both sides of the resulting $q$-identity by $\frac{16}{7(1-q)^3}$, 
 and by then letting $q \to 1^{+}$, 
 this produces the Ramanujan-type formula 
\begin{equation*}
 \frac{15 \Gamma^{2}\left( \frac{1}{4} \right) }{28 \sqrt{2\pi}} = 
 \sum_{n = 0}^{\infty} 
 \left( -\frac{1}{4} \right)^{n} \left[ \begin{matrix} 
 \frac{3}{8}, \frac{1}{2}, \frac{15}{8} 
 \vspace{1mm} \\ 
 \frac{9}{8}, \frac{13}{8}, \frac{7}{4} 
 \end{matrix} \right]_{n} (5 n + 3). 
 \end{equation*}
\end{example}

 \begin{example}
 Setting $(a, b, c, d) = \big( \frac{3}{2}, \frac{1}{6}, 2, 1 \big)$ in Theorem \ref{theoremnegquart}, 
 and by then letting $q > 1$, and by then multiplying both sides of the resulting $q$-identity
 by $\frac{24}{(1-q)^3}$, 
 and by then letting $q \to 1^{+}$, 
 this produces the Ramanujan-type formula 
\begin{equation*}
 \frac{ 2^{\frac{14}{3}} \sqrt{3} \Gamma^{3}\left( \frac{1}{3} \right) }{5 \pi^2} = 
 \sum_{n = 0}^{\infty} 
 \left( -\frac{1}{4} \right)^{n} \left[ \begin{matrix} 
 -\frac{1}{4}, \frac{1}{4}, \frac{5}{6} 
 \vspace{1mm} \\ 
 1, \frac{5}{3}, 2 
 \end{matrix} \right]_{n} (20 n + 17). 
 \end{equation*}
\end{example}

\begin{example}
 Setting $(a, b, c, d) = \big( -\frac{1}{2}, \frac{1}{4}, 1, \frac{1}{2} \big)$ in Theorem \ref{theoremnegquart}, 
 and by then letting $q > 1$, 
 and by then multiplying both sides of the resulting $q$-identity by $\frac{8}{(1-q)^3}$, 
 and by then letting $q \to 1^{+}$, 
 this produces the Ramanujan-type formula 
\begin{equation*}
 \frac{21 \pi^{3/2} }{2 \sqrt{2} \Gamma^{2}\left( \frac{1}{4} \right)} = 
 \sum_{n = 0}^{\infty} 
 \left( -\frac{1}{4} \right)^{n} \left[ \begin{matrix} 
 \frac{1}{2}, \frac{5}{4}, \frac{3}{2} 
 \vspace{1mm} \\ 
 \frac{3}{4}, \frac{11}{8}, \frac{15}{8} 
 \end{matrix} \right]_{n} (5 n + 4). 
 \end{equation*}
\end{example}

\begin{example}
 Setting $(a, b, c, d) = \big( \frac{1}{2}, \frac{1}{6}, 1, 1 \big)$ in Theorem \ref{theoremnegquart}, 
 and by then setting $q > 1$, and by then multiplying both sides of the resulting $q$-identity
 by $\frac{24}{5 (1-q)^3}$, and by then setting $q \to 1^{+}$, 
 this produces the Ramanujan-type formula 
\begin{equation*}
 \frac{2^{11/3} \Gamma^{3}\left( \frac{1}{3} \right)}{ 5 \sqrt{3} \pi^2 } = 
 \sum_{n = 0}^{\infty} 
 \left( -\frac{1}{4} \right)^{n} \left[ \begin{matrix} 
 \frac{1}{4}, \frac{3}{4}, \frac{5}{6} 
 \vspace{1mm} \\ 
 1, 1, \frac{5}{3}
 \end{matrix} \right]_{n} (4 n + 3). 
 \end{equation*}
\end{example}

\subsection{Further $q$-analogues for the convergence rate $\frac{1}{4}$}\label{suqqqqbqqquqqaqrqtqer}
 By applying {\tt EKHAD}-normalization, according to Lemma \ref{lemmaEKHAD}, 
 starting with 
\begin{equation}\label{Fsecondquarter}
 F(n, k;q) := q^{k} \left[ \begin{matrix} 
 q^{n + a}, q^{n+b} \vspace{1mm} \\ 
 q^{n+c}, q^{2n+d} \end{matrix} \, \Bigg| \, q \right]_{k}, 
\end{equation}
 and by mimicking the proof of 
 Theorem \ref{theoremquarter}, 
 the resulting transformation of a ${}_{3}\phi_{2}$-series of the form 
 $$ {}_{3}\phi_{2}\!\!\left[ \begin{matrix} 
 q^{a}, q^{b}, q \vspace{1mm}\\ 
 q^{c}, q^{d} \end{matrix} \ \Bigg| \ q; q 
 \right], $$
 by analogy with Theorems \ref{theoremquarter} 
 and \ref{theoremnegquart} and resulting from a relation of the form 
\begin{equation}\label{makeexplicit1} 
 \sum_{k=0}^{\infty} \widetilde{F}(0, k;q) = \sum_{n=0}^{\infty} \widetilde{G}(n,0;q) 
\end{equation}
 for the $q$-WZ pair $(\widetilde{F}, \widetilde{G})$
 obtained from the application of {\tt EKHAD}-normaliza-tion to 
 \eqref{Fsecondquarter}, this gives a powerful way of obtaining 
 $q$-analogues of accelerated series of convergence rate $\frac{1}{4}$, 
 as demonstrated below. 
 This is given explicitly below. 

\begin{theorem}\label{additionalquarter}
 Letting
 
\noindent $p(n) = q^n \big(q^{a+b+1}-q^{a+d+n+1}-q^{b+d+n+1}-q^{c+d+n}+q^{c+2 d+3 n}+q^{d+1}\big)$, 

 \ 

\noindent the $q$-WZ relation in \eqref{makeexplicit1} gives us that 
\begin{multline*}
 (1-q^d) 
 {}_{3}\phi_{2}\!\!\left[ \begin{matrix} 
 q^{a}, q^{b}, q \vspace{1mm}\\ 
 q^{c}, q^{d} \end{matrix} \ \Bigg| \ q; q 
 \right] \\ 
 = \sum_{n=0}^{\infty} \frac{p(n)}{q^{a+b+1}-q^{n+c+d}}
 \frac{1}{(q^{d+2};q^2)_n (q^{d+1};q^2)_n} 
 \left[ \begin{matrix} 
 q^a, q^b, q^{d-a}, q^{d-b} \vspace{1mm} \\ 
 q^{c}, q^{-a-b+c+d-1} \end{matrix} \, \Bigg| \, q \right]_{n}, 
\end{multline*}
 under the assumption that 
 $q>1$ and that $a,b,c,d\in\mathbb{C}$ satisfy
$
\Re(c)>0$, 
$\Re(d)>0$, and 
$\Re(c+d-a-b)>1$. 
 \end{theorem}

\begin{proof}
 Applying the $q$-version of Zeilberger's algorithm to \eqref{Fsecondquarter}
 produces the rational certificate 

 \ 

\noindent $R(n, k;q) = \frac{1}{1-q^{d+k+2 n}} 
 q^{-k} \big(1-q^{c+n}\big) \big(1-q^{d+2 n}\big) \big(1-q^{d+2 n+1}\big)
 \big(-q^{a+b+k+1}+q^{a+d+k+n+1}+q^{b+d+k+n+1}+q^{c+d+k+n}-q^{c+2 d+2 k+3 n}-q^{d+1}\big)$, 

 \ 

\noindent writing $H(n,k;q) = R(n, k;q) F(n, k;q)$, along with the polynomial coefficients 

 \ 

\noindent $p_{1}(n;q) = -q^2 \left(1-q^{a+n}\right) \left(1-q^{b+n}\right) \left(q^a-q^{d+n}\right) \left(q^b-q^{d+n}\right)$

 \ 

\noindent and 

 \ 

\noindent $p_{2}(n;q) = \big(1-q^{c+n}\big) \big(1-q^{d+2 n}\big) \big(1-q^{d+2 n+1}\big) \big(q^{a+b+1}-q^{c+d+n}\big)$, 

 \ 

\noindent so that the $q$-difference equation in \eqref{generalqZeil} 
 holds. Applying, via Lemma \ref{lemmaEKHAD}, {\tt EKHAD}-normalization, 
 this gives rise to the normalized version 
$$ \widetilde{F}(n,k;q) 
 = \frac{q^{k+n}}{ \left[ \begin{matrix} 
 q^d, q^{d+1} \end{matrix} \, \big| \, q^2 \right]_{n} } 
 \left[ \begin{matrix} 
 q^a, q^b, q^{d-a}, q^{d-b} \vspace{1mm} \\ 
 q^{-a-b+c+d-1}, q^c \end{matrix} \, \Bigg| \, q \right]_{n} 
 \left[ \begin{matrix} 
 q^{a+n}, q^{b+n} \vspace{1mm} \\ 
 q^{c+n}, q^{d+2 n} \end{matrix} \, \Bigg| \, q \right]_{k} $$
 of the input function in \eqref{Fsecondquarter}. 
 The associated certificate is 

 \ 

\noindent $ \widetilde{R}(n,k;q)
 = \frac{q^{-k} }{\left(q^{d+k+2 n}-1\right)
 \left(q^{c+d+n}-q^{a+b+1}\right)} \big(q^{a+b+k+1}-q^{a+d+k+n+1}-q^{b+d+k+n+1}-q^{c+d+k+n}+q^{c+2 d+2 k+3 n}+q^{d+1}\big)$, 

 \ 

\noindent and, by writing $\widetilde{G}(n,k;q) = \widetilde{R}(n,k;q)\widetilde{F}(n,k;q)$, 
 we find that the $q$-WZ difference equation in \eqref{qWZequation} is satisfied.

If a numerator $q$-shifted factorial terminates, then the relevant series
 is finite and the required convergence is immediate. Thus, the following ratio-test argument concerns
 only the nonterminating case. 
In the nonterminating case, direct ratio estimates from the displayed formulas for
$\widetilde{F}$ and $\widetilde{G}=\widetilde{R}\widetilde{F}$ show, for every fixed
admissible $(a,b,c,d)$, that $\widetilde{F}(N,k;q)\to0$ as $N\to\infty$ for each
fixed $k$, that $\sum_{k=0}^{\infty}|\widetilde{F}(0,k;q)|$ converges and
$\sum_{n=0}^{\infty}|\widetilde{G}(n,k;q)|$ converges for each fixed $k$, and that
there exist $k_{0}\geq0$, $M>0$, $\rho\in(0,1)$, and a nonnegative sequence
$(A_n)_{n\geq0}$ with $\sum_{n=0}^{\infty}A_n<\infty$ such that
$|\widetilde{G}(n,k;q)|\leq M\rho^{\,k-k_{0}}A_n$ for all $n\geq0$ and
$k\geq k_{0}$; hence all the hypotheses of
Lemma~\ref{lemmaRectangularTelescope} are satisfied. 

Therefore, Lemma \ref{lemmaRectangularTelescope} yields the stated $q$-series identity. 
 Moreover, the expression $\widetilde{F}(0,k;q)$ is the $k^{\text{th}}$ summand
 of the displayed ${}_{3}\phi_{2}$-series, 
 while $(1-q^d) \widetilde{G}(n,0;q)$ simplifies to the $n^\text{th}$ summand on the right-hand side of 
 Theorem \ref{additionalquarter}. 
 \end{proof}

\begin{example}
 Setting $(a, b, c, d) = \big( 1, 1, \frac{3}{2}, 2 \big)$ in 
 Theorem \ref{additionalquarter}, 
 we obtain the second $q$-identity introduced
 in Section \ref{subFabryGuillera}. 
\end{example}

\begin{example}
 Setting $(a, b, c, d) = \big( \frac{1}{4}, \frac{1}{4}, 1, \frac{3}{4} \big)$ 
 in Theorem \ref{additionalquarter}, 
 and by then letting $q > 1$, and by then multiplying both sides of the resulting $q$-identity by $\frac{1}{1-q}$, 
 and by then letting $q \to 1^{+}$, 
 this produces the Ramanujan-type formula 
 $$ \frac{3\sqrt{2}}{4} = \sum_{n = 0}^{\infty} 
 \left( \frac{1}{4} \right)^{n} \left[ \begin{matrix} 
 \frac{1}{4}, \frac{1}{2}, \frac{1}{2} \vspace{1mm} \\ 
 \frac{7}{8}, 1, \frac{11}{8} \end{matrix} \right]_{n} (3n+1). $$
\end{example}

\begin{example}
 Setting $(a, b, c, d) = \big( \frac{1}{6}, \frac{1}{2}, 1, \frac{5}{6} \big)$ 
 in Theorem \ref{additionalquarter}, 
 and by then letting $q > 1$ and multiplying both sides of the resulting $q$-identity by $\frac{3}{1-q}$,
 and by then letting $q \to 1^{+}$, 
 this produces the Ramanujan-type formula 
 $$ \frac{5 \sqrt{3}}{2} = \sum_{n = 0}^{\infty} 
 \left( \frac{1}{4} \right)^{n} \left[ \begin{matrix} 
 \frac{1}{3}, \frac{1}{2}, \frac{2}{3} \vspace{1mm} \\ 
 \frac{11}{12}, 1, \frac{17}{12} \end{matrix} \right]_{n} (9 n + 4). $$
\end{example} 

\begin{example}
 Setting $(a, b, c, d) = \big( \frac{1}{3}, \frac{5}{6}, 1, \frac{3}{2} \big)$ 
 in Theorem \ref{additionalquarter}, 
 and by then letting $q > 1$ and multiplying both sides of the resulting $q$-identity by $\frac{3}{1-q}$,
 and by then letting $q \to 1^{+}$, 
 this produces the Ramanujan-type formula 
 $$ \frac{27 \sqrt[3]{2} }{4} = \sum_{n = 0}^{\infty} 
 \left( \frac{1}{4} \right)^{n} \left[ \begin{matrix} 
 \frac{2}{3}, \frac{5}{6}, \frac{7}{6} \vspace{1mm} \\ 
 1, \frac{5}{4}, \frac{7}{4} \end{matrix} \right]_{n} (9 n + 7). $$
\end{example} 

\begin{example}
 Setting $(a, b, c, d) = \big( \frac{1}{3}, \frac{5}{6}, \frac{3}{2}, 1 \big)$ 
 in Theorem \ref{additionalquarter}, 
 and by then letting $q>1$ and multiplying both sides of the resulting $q$-identity by $\frac{6}{1-q}$, 
 and by then letting $q \to 1^{+}$, 
 we obtain the Ramanujan-type formula 
 $$ 9 \sqrt[3]{2} = \sum_{n = 0}^{\infty} 
 \left( \frac{1}{4} \right)^{n} \left[ \begin{matrix} 
 \frac{1}{6}, \frac{2}{3}, \frac{5}{6} \vspace{1mm} \\ 
 1, \frac{3}{2}, \frac{3}{2} \end{matrix} \right]_{n} (18 n + 11). $$
\end{example} 

\begin{example}
 Setting $(a, b, c, d) = \big( \frac{1}{6}, \frac{1}{3}, 1, \frac{5}{6} \big)$ 
 in Theorem \ref{additionalquarter}, 
 and by then letting $q > 1$ and multiplying both sides of the resulting
 $q$-identity by $\frac{1}{1-q}$ and by then letting $q \to 1^{+}$, 
 this produces the Ramanujan-type formula 
 $$ \frac{5 \sqrt[3]{2} }{6} = \sum_{n = 0}^{\infty} 
 \left( \frac{1}{4} \right)^{n} \left[ \begin{matrix} 
 \frac{1}{6}, \frac{1}{2}, \frac{2}{3} \vspace{1mm} \\ 
 \frac{11}{12}, 1, \frac{17}{12} \end{matrix} \right]_{n} (3 n + 1). $$
\end{example} 

\begin{example}
 Setting $(a, b, c, d) = \big( \frac{1}{2}, -\frac{1}{2}, \frac{3}{4}, 1 \big)$ 
 in Theorem \ref{additionalquarter}, 
 and by then letting $q > 1$ and by multiplying both sides of the resulting $q$-identity
 by $\frac{3}{2(1-q)}$, 
 and by then letting $q \to 1^{+}$, 
 this yields the Ramanujan-type formula 
 $$ \frac{12\pi^2}{\Gamma^{4}\left( \frac{1}{4} \right)} = \sum_{n = 0}^{\infty} 
 \left( \frac{1}{4} \right)^{n} \left[ \begin{matrix} 
 -\frac{1}{2}, \frac{1}{2}, \frac{3}{2} \vspace{1mm} \\ 
 \frac{3}{4}, 1, \frac{7}{4} \end{matrix} \right]_{n} (3n+1). $$
\end{example} 

\begin{example}
 Setting $(a, b, c, d) = \big( \frac{1}{2}, -\frac{1}{2}, \frac{5}{6}, 1 \big)$ 
 in Theorem \ref{additionalquarter}, 
 and by then letting $q > 1$ and by multiplying both sides of the 
 resulting $q$-identity by $\frac{10}{1-q}$, 
 and by then letting $q \to 1^{+}$, 
 this gives us the Ramanujan-type formula 
 $$ \frac{ 5 \, 2^{\frac{11}{3}} \pi^3 }{ \Gamma^{6}\left( \frac{1}{3} 
 \right) } = \sum_{n = 0}^{\infty} 
 \left( \frac{1}{4} \right)^{n} \left[ \begin{matrix} 
 -\frac{1}{2}, \frac{1}{2}, \frac{3}{2} \vspace{1mm} \\ 
 \frac{5}{6}, 1, \frac{11}{6} \end{matrix} \right]_{n} (18n+7). $$
\end{example} 

\begin{example}
 Setting $(a, b, c, d) = \big( \frac{1}{2}, \frac{5}{6}, 2, 1 \big)$ 
 in Theorem \ref{additionalquarter}, 
 and by then letting $q > 1$, and by then multiplying both sides of the resulting $q$-identity 
 by $\frac{8}{1-q}$, and by then letting $q \to 1^{+}$, 
 this produces the Ramanujan-type formula 
 $$ \frac{ 2^{\frac{19}{3}} \pi }{ \Gamma^{3}\left( \frac{1}{3} 
 \right) } = \sum_{n = 0}^{\infty} 
 \left( \frac{1}{4} \right)^{n} \left[ \begin{matrix} 
 \frac{1}{6}, \frac{1}{2}, \frac{5}{6} \vspace{1mm} \\ 
 1, \frac{5}{3}, 2 \end{matrix} \right]_{n} (18n+13). $$
\end{example} 

\begin{example}
 Setting $(a, b, c, d) = \big( -\frac{1}{2}, \frac{1}{2}, \frac{1}{3}, 1 \big)$ in Theorem \ref{additionalquarter}, 
 and by then letting $q > 1$, and by then multiplying both sides of the resulting $q$-identity by 
 $\frac{4}{1-q}$, 
 and by then letting $q \to 1^{+}$, 
 this produces the Ramanujan-type formula 
 $$ -\frac{\Gamma^{6}\left( \frac{1}{3} \right)}{2^{\frac{5}{3}} \pi^3} = \sum_{n = 0}^{\infty} 
 \left( \frac{1}{4} \right)^{n} \left[ \begin{matrix} 
 -\frac{1}{2}, \frac{1}{2}, \frac{3}{2} \vspace{1mm} \\ 
 \frac{1}{3}, 1, \frac{4}{3} \end{matrix} \right]_{n} (18n + 1). $$
\end{example}

\begin{example} 
 Setting $(a, b, c, d) = \big( -\frac{1}{2}, \frac{1}{3}, 1, \frac{1}{6} \big)$ 
 in Theorem \ref{additionalquarter}, 
 and by then letting $q > 1$, and by then multiplying both sides of the resulting $q$-identity by $\frac{3}{1-q}$, and by then letting $q \to 1^{+}$, this gives us the Ramanujan-type formula $$ -\frac{ \sqrt{3} \Gamma^{6}\left( \frac{1}{3} \right) }{2^{\frac{14}{3}} \pi^3} = \sum_{n = 0}^{\infty} \left( \frac{1}{4} \right)^{n} \left[ \begin{matrix} -\frac{1}{2}, -\frac{1}{6}, \frac{2}{3} \vspace{1mm} \\ \frac{7}{12}, 1, \frac{13}{12} \end{matrix} \right]_{n} (9n - 1). $$ \end{example}

\subsection{$q$-analogues for the convergence rate $\frac{1}{64}$}
 Applying {\tt EKHAD}-normalization, according to Lemma \ref{lemmaEKHAD}, to 
\begin{equation}\label{Ffor64}
 F(n, k;q) := q^{k} \left[ \begin{matrix} 
 q^{n+a}, q^{n+b} \vspace{1mm} \\ 
 q^{2n+c}, q^{2n+d} \end{matrix} \, \Bigg| \, q \right]_{k}, 
\end{equation}
 by analogy with the preceding subsection, this gives us 
 a way of obtaining $q$-analogues of 
 accelerated series of convergence rate $\frac{1}{64}$, as below. 
 This is made explicit below. 

\begin{theorem}\label{theorem1over64}
 Letting 

 \ 

\noindent $p(n) = -q^n \big(-q^{a+b+c+d+2 n}+q^{a+b+2 c+d+4 n}+q^{a+b+2 c+d+4 n+1} + 
 q^{a+b+c+2 d+4 n}+q^{a+b+c+2 d+4 n+1}-q^{a+b+2 c+2 d+6 n}-q^{2
 a+b+c+n+1} - q^{a+2 b+c+n+1}+ q^{a+b+c+1}-$

\noindent $ q^{2 a+b+d+n+1}-q^{a+2 b+d+n+1}+q^{a+b+d+1}+q^{2 a+2 b+1}-q^{a+c+d+n+1}+q^{2 a+c+d+2 n+1}-q^{a+2 c+2
 d+5 n+1}-q^{b+c+d+n+1}+q^{2 b+c+d+2 n+1}-q^{b+2 c+2 d+5 n+1}+q^{2 c+2 d+4 n}+ $

\noindent $q^{2 c+2 d+4 n+1}-q^{3 c+2 d+6 n}-q^{2 c+3 d+6 n}+q^{3 c+3 d+8
 n}\big),$ 

 \ 

\noindent we have that 
\begin{multline*}
 (1 - q^c) (1-q^d) 
 {}_{3}\phi_{2}\!\!\left[ \begin{matrix} 
 q^{a}, q^{b}, q \vspace{1mm}\\ 
 q^{c}, q^{d} \end{matrix} \ \Bigg| \ q; q 
 \right] 
 = \\ 
 \sum_{n=0}^{\infty} \frac{p(n)}{\left(q^{a+b}-q^{c+d+2 n}\right) \left(q^{c+d+2 n}-q^{a+b+1}\right)} \times \\ 
 \frac{ \left[ \begin{matrix} 
 q^a, q^b, q^{c-a}, q^{c-b}, q^{d-a}, q^{d-b} \end{matrix} \, \big| \, q \right]_{n} }{ 
 \left[ \begin{matrix} 
 q^{c+1}, q^{c+2}, q^{d+1}, 
 q^{d+2}, q^{-a-b+c+d-1}, q^{-a-b+c+d} \end{matrix} \, \big| \, q^2 \right]_{n} }, 
\end{multline*}
 under the assumption that 
 $q>1$ and that $a,b,c,d\in\mathbb{C}$ satisfy $\Re(c)>0$, $\Re(d)>0$, and $\Re(c+d-a-b)>1$. 
\end{theorem}

\begin{proof}
 Applying the $q$-version of Zeilberger's algorithm to 
 \eqref{Ffor64}, this produces the rational certificate 

 \ 

\noindent $R(n,k;q) = -\frac{1}{\big(1-q^{c+k+2 n}\big) \big(1-q^{d+k+2 n}\big)}
 q^{-k} \big(q^{c+2 n}-1\big) \big(q^{c+2 n+1}-1\big) \big(q^{d+2 n}-1\big) $

\noindent $ \big(q^{d+2 n+1}-1\big) \big(-q^{a+b+c+d+k+2 n}+q^{a+b+2
 c+d+2 k+4 n}+q^{a+b+2 c+d+2 k+4 n+1} + $

\noindent $q^{a+b+c+2 d+2 k+4 n}+q^{a+b+c+2 d+2 k+4 n+1}-q^{a+b+2 c+2 d+3 k+6 n}-q^{2 a+b+c+k+n+1} - $

\noindent $ q^{a+2
 b+c+k+n+1}+q^{a+b+c+1}-q^{2 a+b+d+k+n+1}-q^{a+2 b+d+k+n+1}+q^{a+b+d+1} + $

\noindent $ q^{2 a+2 b+k+1}+q^{2 a+c+d+k+2 n+1}-q^{a+2 c+2 d+2 k+5
 n+1}-q^{a+c+d+n+1}+q^{2 b+c+d+k+2 n+1}-q^{b+2 c+2 d+2 k+5 n+1}-q^{b+c+d+n+1}+q^{2 c+2 d+k+4 n}+q^{2 c+2 d+k+4 n+1}-q^{2 c+3 d+2 k+6 n}+q^{3 c+3
 d+3 k+8 n}-q^{3 c+2 (d+k+3 n)}\big) $, 

 \ 

\noindent writing $H(n,k;q) = R(n,k;q) F(n,k;q)$, along with the polynomial coefficients 

 \ 

\noindent $p_1(n;q) = -q^2 \big(1-q^{a+n}\big) \big(1-q^{b+n}\big) \big(q^a-q^{c+n}\big) \big(q^a-q^{d+n}\big) \big(q^b-q^{c+n}\big)
 \big(q^b-q^{d+n}\big)$ 

 \ 

\noindent and 

 \ 

\noindent $p_2(n;q) = \big(1-q^{c+2 n}\big) \big(1-q^{c+2 n+1}\big) \big(1-q^{d+2 n}\big) \big(1-q^{d+2 n+1}\big) \big(q^{a+b}-q^{c+d+2 n}\big) $

\noindent $ \big(q^{a+b+1}-q^{c+d+2 n}\big)$, 

 \ 

\noindent giving us that the $q$-difference equation in \eqref{generalqZeil} 
 holds true. 
 Applying {\tt EKHAD}-normalization to 
 \eqref{Ffor64}, via Lemma \ref{lemmaEKHAD}, we obtain the normalized version 
\begin{multline*}
 \widetilde{F}(n,k;q) 
 = q^{k+n} 
 \left[ \begin{matrix} 
 q^{a+n}, q^{b+n} \vspace{1mm} \\ 
 q^{c+2 n}, q^{d+2 n} \end{matrix} \, \Bigg| \, q \right]_{k} \times \\ 
 \frac{ \left[ \begin{matrix} 
 q^a, q^b, q^{-a+c}, q^{-b+c}, q^{-a+d}, q^{-b+d} \end{matrix} \, \big| \, q \right]_{n} }{ 
 \left[ \begin{matrix} 
 q^c, q^{c+1}, q^d, q^{d+1}, q^{-a-b+c+d-1}, q^{-a-b+c+d} \end{matrix} \, \big| \, q^2 \right]_{n} }
\end{multline*}
 of the input function in \eqref{Ffor64}. 
 This, in turn, gives rise, via the $q$-Zeilberger algorithm, 
 to the certificate

 \ 

\noindent $\widetilde{R}(n,k;q) = \frac{q^{-k}}{\left(1-q^{c+k+2 n}\right) \left(1-q^{d+k+2 n}\right) \left(q^{a+b}-q^{c+d+2 n}\right) \left(q^{a+b+1}-q^{c+d+2 n}\right)} \times $

\noindent $ \big( -q^{a+b+c+d+k+2 n}+q^{a+b+2 c+d+2 k+4 n}+q^{a+b+2 c+d+2 k+4 n+1}+q^{a+b+c+2 d+2 k+4 n}+q^{a+b+c+2 d+2 k+4 n+1}-q^{a+b+2 c+2 d+3 k+6 n}-q^{2
 a+b+c+k+n+1}-q^{a+2 b+c+k+n+1} + $

\noindent $ q^{a+b+c+1}-q^{2 a+b+d+k+n+1}-q^{a+2 b+d+k+n+1}+q^{a+b+d+1}+q^{2 a+2 b+k+1} + $

\noindent $q^{2 a+c+d+k+2 n+1}-q^{a+2 c+2 d+2
 k+5 n+1}-q^{a+c+d+n+1}+q^{2 b+c+d+k+2 n+1} - $

\noindent $ q^{b+2 c+2 d+2 k+5 n+1}-q^{b+c+d+n+1}+q^{2 c+2 d+k+4 n}+q^{2 c+2 d+k+4 n+1}-q^{2 c+3 d+2 k+6 n}+q^{3
 c+3 d+3 k+8 n}-q^{3 c+2 (d+k+3 n)} \big).$

 \ 

\noindent By then writing $\widetilde{G}(n,k;q) = \widetilde{R}(n,k;q) \widetilde{F}(n,k;q)$, 
 we find that the $q$-WZ difference equation in \eqref{qWZequation} holds.

If a numerator $q$-shifted factorial terminates, then the relevant series
 is finite and the required convergence is immediate. Thus, the following ratio-test argument concerns
 only the nonterminating case. 
In the nonterminating case, direct ratio estimates from the displayed formulas for
$\widetilde{F}$ and $\widetilde{G}=\widetilde{R}\widetilde{F}$ show, for every fixed
admissible $(a,b,c,d)$, that $\widetilde{F}(N,k;q)\to0$ as $N\to\infty$ for each
fixed $k$, that $\sum_{k=0}^{\infty}|\widetilde{F}(0,k;q)|$ converges and
$\sum_{n=0}^{\infty}|\widetilde{G}(n,k;q)|$ converges for each fixed $k$, and that
there exist $k_{0}\geq0$, $M>0$, $\rho\in(0,1)$, and a nonnegative sequence
$(A_n)_{n\geq0}$ with $\sum_{n=0}^{\infty}A_n<\infty$ such that
$|\widetilde{G}(n,k;q)|\leq M\rho^{\,k-k_{0}}A_n$ for all $n\geq0$ and
$k\geq k_{0}$; hence all the hypotheses of
Lemma~\ref{lemmaRectangularTelescope} are satisfied.

Therefore, Lemma \ref{lemmaRectangularTelescope} yields the stated $q$-series identity. 
 Also, the expression $\widetilde{F}(0,k;q)$ is the $k^{\text{th}}$ summand 
 of the displayed ${}_{3}\phi_{2}$-series, 
 and $(1-q^c) (1-q^d) \widetilde{G}(n,0;q)$ simplifies to the $n^{\text{th}}$
 summand on the right-hand side of Theorem \ref{theorem1over64}. 
 \end{proof}

\begin{example}
 Setting $(a,b,c,d) = \big( 1, 1, 2, 2 \big)$ in Theorem \ref{theorem1over64}, 
 this provides the $q$-identity 
 highlighted in Section \ref{sectionZeilberger}.
\end{example}

\begin{example}
 Setting $(a,b,c,d) = \big( \frac{1}{2}, \frac{1}{2}, 1, \frac{3}{2} \big)$
 in Theorem \ref{theorem1over64}, 
 and letting

 \ 

\noindent $ r(n) = \big( 3 q^{2 n+2}+q^{4 n+1}+q^{4 n+3}+q^{6 n+2}+q^{6 n+4}-2 q^{8 n+4}-2 q^{8 n+6}-2 q^{10 n+5}+q^{12 n+7}+q^{14 n+8}-q-2
 \big)/\big( q^{2 n+1}+1 \big), $ 

 \ 

\noindent we obtain the $q$-identity 
\begin{multline*}
 -(1-q)^2 (1-q^3)^{2} \sum_{k=0}^{\infty} \frac{q^{2k}}{1-q^{2k+1}} 
 \left[ \begin{matrix} 
 q \vspace{1mm} \\ 
 q^2 \end{matrix} \, \Bigg| \, q^2 \right]_{k} = \\ 
 \sum_{n=0}^{\infty} q^{2n} \frac{ \left( q;q^{2} \right)_{n}^{4} 
 \left( q^{2};q^{2} \right)_{n}^{2} }{ \left( q^{2};q^{4} \right)_{n} 
 \left( q^{4};q^{4} \right)_{n} 
 \left( q^{5};q^{4} \right)_{n}^{2} \left( q^{7};q^{4} \right)_{n}^{2} } r(n).
\end{multline*} 
 By then letting $q > 1$, and
 by then multiplying both sides of the resulting $q$-identity by 
 $\frac{1}{2(q-1)^3}$, 
 and by then letting $q \to 1^{+}$, 
 we obtain the Chu formula 
\begin{equation*}
 \frac{9\pi}{4} = 
 \sum_{n = 0}^{\infty} 
 \left( \frac{1}{64} \right)^{n} \left[ \begin{matrix} 
 1, \frac{1}{2}, \frac{1}{2}, \frac{1}{2} 
 \vspace{1mm} \\ 
 \frac{5}{4}, \frac{5}{4}, \frac{7}{4}, \frac{7}{4}
 \end{matrix} \right]_{n} \left( 42n^3 + 75n^2+42n + 7 \right) 
 \end{equation*} 
 introduced via well-poised $\Omega$-sums \cite{Chu2021Omega}. 
\end{example}

\begin{example}
 Setting $(a,b,c,d) = \big( \frac{1}{2}, \frac{1}{6}, 1, 1 \big)$ in Theorem \ref{theorem1over64}, 
 and by then letting $q > 1$ and by multiplying both sides of the resultant 
 $q$-identity by $\frac{64}{(1-q)^2}$, 
 and by then letting $q \to 1^{+}$, 
 this yields the Ramanujan-type formula 
\begin{equation*}
 \frac{ 2^{\frac{14}{3}} \sqrt{3} \Gamma^{3}\left( \frac{1}{3} \right) }{\pi^2} = 
 \sum_{n = 0}^{\infty} 
 \left( \frac{1}{64} \right)^{n} \left[ \begin{matrix} 
 \frac{1}{2}, \frac{5}{6}, \frac{5}{6} 
 \vspace{1mm} \\ 
 1, 1, \frac{5}{3} 
 \end{matrix} \right]_{n} \left( 126 n + 85 \right). 
 \end{equation*} 
\end{example}

\subsection{$q$-analogues for the convergence rate $-\frac{1}{27}$}
 By applying {\tt EKHAD}-normalization, according to Lemma \ref{lemmaEKHAD}, to 
\begin{equation}\label{inputfor27}
 F(n,k;q) = q^{k} \left[ \begin{matrix} 
 q^{a}, q^{n+b} \vspace{1mm} \\ 
 q^{2n+c}, q^{2n+d} \end{matrix} \, \Bigg| \, q \right]_{k}, 
\end{equation}
 by analogy with the preceding section, this gives us a way of generating $q$-analogues of accelerated series of convergence rate 
 $-\frac{1}{27}$, as below. 
 This is made explicit below

\begin{theorem}\label{theorem1over27}
 Letting 

 \

\noindent $p(n) = -q^{3 a+2 b+1}+q^{3 a+b+c+n+1}+q^{3 a+b+d+n+1}-q^{3 a+c+d+2 n+1}-q^{2 a+b+2 c+3 n+1}- $

\noindent $q^{2 a+b+c+d+3 n+1}-q^{2 a+b+2 d+3 n+1}+q^{2 a+2 b+2 c+4
 n+1}+q^{2 a+2 b+c+d+4 n} + $

\noindent $2 q^{2 a+2 b+c+d+4 n+1}+q^{2 a+2 b+c+d+4 n+2}+q^{2 a+2 c+d+4 n+1}+q^{2 a+2 b+2 d+4 n+1} + $

\noindent $q^{2 a+c+2 d+4 n+1}+q^{a+b+2
 c+d+5 n+1}+q^{a+b+2 c+d+5 n+2}-q^{2 a+b+2 c+d+5 n+1} - $

\noindent $q^{2 a+b+2 c+d+5 n+2}+q^{a+b+c+2 d+5 n+1}+q^{a+b+c+2 d+5 n+2}-q^{2 a+b+c+2 d+5 n+1}-$

\noindent $q^{2
 a+b+c+2 d+5 n+2}-q^{a+2 b+2 c+d+6 n+1}-q^{a+2 b+2 c+d+6 n+2}-q^{2 a+2 b+2 c+d+6 n}-$

\noindent $q^{2 a+2 b+2 c+d+6 n+1}-q^{a+2 b+c+2 d+6 n+1}-q^{a+2 b+c+2
 d+6 n+2}-q^{2 a+2 b+c+2 d+6 n}-$

\noindent $q^{2 a+2 b+c+2 d+6 n+1}-q^{a+2 c+2 d+6 n+1}-q^{a+2 c+2 d+6 n+2}+q^{2 a+2 c+2 d+6 n+2}-$

\noindent $q^{b+2 c+2 d+7
 n+2}-q^{a+b+2 c+2 d+7 n}-q^{a+b+2 c+2 d+7 n+1}+q^{2 a+b+2 c+2 d+7 n+1}+$

\noindent $q^{2 b+2 c+2 d+8 n+2}+q^{a+2 b+2 c+2 d+8 n+1}+q^{2 a+2 b+2 c+2 d+8
 n}+q^{a+b+3 c+2 d+9 n}+$

\noindent $q^{a+b+3 c+2 d+9 n+1}+q^{a+b+3 c+2 d+9 n+2}+q^{a+b+2 c+3 d+9 n}+q^{a+b+2 c+3 d+9 n+1}+$

\noindent $q^{a+b+2 c+3 d+9 n+2}+q^{3 c+3
 d+10 n+1}+q^{3 c+3 d+10 n+2}-q^{a+3 c+3 d+10 n+2}-q^{b+3 c+3 d+11 n+2}-q^{a+b+3 c+3 d+11 n}-q^{a+b+3 c+3 d+11 n+1}-q^{4 c+3 d+12 n+1}-q^{3
 c+4 d+12 n+1}+q^{4 c+4 d+14 n+1}$,

 \ 

\noindent we obtain the $q$-identity 
\begin{multline*}
 (1 - q^c) (1-q^d) 
 {}_{3}\phi_{2}\!\!\left[ \begin{matrix} 
 q^{a}, q^{b}, q \vspace{1mm}\\ 
 q^{c}, q^{d} \end{matrix} \ \Bigg| \ q; q 
 \right] \\ 
 = \sum_{n=0}^{\infty} 
 \frac{(-1)^n p(n) q^{\frac{1}{2} n (2 a-2 b-n+1)-n}}{\left(q^{a+b}-q^{c+d+3 n}\right) \left(q^{c+d+3 n}-q^{a+b+1}\right) \left(q^{a+b}-q^{c+d+3
 n+1}\right)} \times \\ 
 \left[ \begin{matrix} 
 q^{-a+c}, q^{-a+c+1}, q^{-a+d}, q^{-a+d+1} \vspace{1mm} \\ 
 q^{c+1}, q^{c+2}, q^{d+1}, q^{d+2} \end{matrix} \, \Bigg| \, q^2 \right]_{n} \times \\ 
 \frac{ \left[ \begin{matrix} 
 q^b, q^{-b+c}, q^{-b+d} \end{matrix} \, \big| \, q \right]_{n} }{ 
 \left[ \begin{matrix} 
 q^{-a-b+c+d-1}, q^{-a-b+c+d}, q^{-a-b+c+d+1} \end{matrix} \, \big| \, q^3 \right]_{n} }, 
\end{multline*}
 under the assumptions 
 that $q>1$ and that $a,b,c,d\in\mathbb{C}$ satisfy $\Re(c)>0$, $\Re(d)>0$, and $\Re(c+d-a-b)>1$. 
\end{theorem}

\begin{proof}
 Applying the $q$-Zeilberger algorithm to the input function in \eqref{inputfor27}, 
 this gives rise to the certificate 

 \ 

\noindent $R(n,k;q) = \frac{1}{\big(1-q^{c+k+2 n}\big) \big(1-q^{d+k+2 n}\big)}
 q^{1-k} \big(1-q^{c+2 n}\big) \big(1-q^{c+2 n+1}\big) \big(1-q^{d+2 n}\big) \big(1-q^{d+2 n+1}\big) \big(-q^{3 a+2 b+1}+q^{3
 a+b+c+n+1}+q^{3 a+b+d+n+1}-q^{3 a+c+d+2 n+1}-q^{2 a+b+2 c+3 n+1}-q^{2 a+b+c+d+3 n+1}-q^{2 a+b+2 d+3 n+1}+q^{2 a+2 c+d+4 n+1}+q^{2 a+c+2 d+4
 n+1} + $

\noindent $ q^{2 a+2 b+2 c+k+4 n+1}+q^{2 a+2 b+c+d+k+4 n}+2 q^{2 a+2 b+c+d+k+4 n+1}+q^{2 a+2 b+c+d+k+4 n+2}+q^{2 a+2 b+2 d+k+4 n+1}+q^{a+b+2 c+d+5
 n+1}+q^{a+b+2 c+d+5 n+2}+q^{a+b+c+2 d+5 n+1} + $

\noindent $q^{a+b+c+2 d+5 n+2}-q^{2 a+b+2 c+d+k+5 n+1}-q^{2 a+b+2 c+d+k+5 n+2}-q^{2 a+b+c+2 d+k+5 n+1}-q^{2
 a+b+c+2 d+k+5 n+2}-q^{a+2 c+2 d+6 n+1}-q^{a+2 b+2 c+d+k+6 n+1}-q^{a+2 b+2 c+d+k+6 n+2}-q^{a+2 b+c+2 d+k+6 n+1}-q^{a+2 b+c+2 d+k+6 n+2}+q^{2 a+2
 c+2 d+k+6 n+2}-q^{2 a+2 b+2 c+d+2 k+6 n}-q^{2 a+2 b+2 c+d+2 k+6 n+1}-q^{2 a+2 b+c+2 d+2 k+6 n}-q^{2 a+2 b+c+2 d+2 k+6 n+1}-q^{b+2 c+2 d+7
 n+2}-q^{a+b+2 c+2 d+k+7 n}-q^{a+b+2 c+2 d+k+7 n+1}+q^{2 a+b+2 c+2 d+2 k+7 n+1}+q^{2 b+2 c+2 d+k+8 n+2}+q^{a+2 b+2 c+2 d+2 k+8 n+1}+q^{2 a+2 b+2
 c+2 d+3 k+8 n}+q^{a+b+3 c+2 d+2 k+9 n} + $

\noindent $q^{a+b+3 c+2 d+2 k+9 n+1}+q^{a+b+3 c+2 d+2 k+9 n+2}+q^{a+b+2 c+3 d+2 k+9 n} + $

\noindent $q^{a+b+2 c+3 d+2 k+9
 n+1}+q^{a+b+2 c+3 d+2 k+9 n+2}+q^{3 c+3 d+k+10 n+1}+q^{3 c+3 d+k+10 n+2}-q^{a+3 c+3 d+2 k+10 n+2}-q^{b+3 c+3 d+2 k+11 n+2}-q^{a+b+3 c+3 d+3
 k+11 n} - $

\noindent $ q^{a+b+3 c+3 d+3 k+11 n+1}-q^{4 c+3 d+2 k+12 n+1}-q^{3 c+4 d+2 k+12 n+1}+q^{4 c+4 d+3 k+14 n+1}-q^{a+2 (c+d+3 n+1)}\big), $

 \ 

\noindent writing $H(n,k;q) = R(n,k;q) F(n,k;q)$, along with the polynomial coefficients 

 \ 

\noindent $p_{1}(n;q) 
 = q^2 \big(1-q^{b+n}\big) \big(q^a-q^{c+2 n}\big) \big(q^{c+2 n+1}-q^a\big) \big(q^a-q^{d+2 n}\big) \big(q^{d+2 n+1}-q^a\big)
 \big(q^b-q^{c+n}\big) \big(q^b-q^{d+n}\big)$

 \ 

\noindent and 

 \ 

\noindent $p_{2}(n;q) = q^{n+1} \big(1-q^{c+2 n}\big) \big(1-q^{c+2 n+1}\big) \big(1-q^{d+2 n}\big) \big(1-q^{d+2 n+1}\big) \big(q^{a+b+1}-q^{c+d+3 n}\big)
 \big(q^{c+d+3 n}-q^{a+b}\big) \big(q^{c+d+3 n+1}-q^{a+b}\big), $

 \ 

\noindent so that the $q$-difference equation in \eqref{generalqZeil} is satisfied. 
 Applying {\tt EKHAD}-normali-zation, according to Lemma \ref{lemmaEKHAD}, 
 this gives rise to the normalized version 
\begin{multline*}
 \widetilde{F}(n,k;q) = (-1)^n q^{\frac{1}{2} n (2 a-2 b-n+1)+k} 
 \left[ \begin{matrix} 
 q^{-a+c}, q^{-a+c+1}, q^{-a+d}, q^{-a+d+1} \vspace{1mm} \\ 
 q^c, q^{c+1}, q^d, q^{d+1} \end{matrix} \, \Bigg| \, q^2 \right]_{n} \times \\ 
 \left[ \begin{matrix} 
 q^a, q^{b+n} \vspace{1mm} \\ 
 q^{c+2 n}, q^{d+2 n} \end{matrix} \, \Bigg| \, q \right]_{k} 
 \frac{ \left[ \begin{matrix} 
 q^b, q^{-b+c}, q^{-b+d} \end{matrix} \, \big| \, q \right]_{n} }{ 
 \left[ \begin{matrix} 
 q^{-a-b+c+d-1}, q^{-a-b+c+d}, q^{-a-b+c+d+1} \end{matrix} \, \big| \, q^3 \right]_{n} }
\end{multline*}
 of the input function in \eqref{inputfor27}. 
 In turn, again via the $q$-Zeilberger algorithm, 
 this gives rise to the 
 certificate 

 \ 

\noindent $\widetilde{R}(n,k;q) = \frac{1}{\left(1-q^{c+k+2 n}\right) \left(1-q^{d+k+2 n}\right) \left(q^{c+d+3 n}-q^{a+b}\right) \left(q^{c+d+3 n}-q^{a+b+1}\right) \left(q^{c+d+3
 n+1}-q^{a+b}\right)} $

\noindent $ q^{-k-n} \big(-q^{3 a+2 b+1}+q^{3 a+b+c+n+1}+q^{3 a+b+d+n+1}-q^{3 a+c+d+2 n+1}-q^{2 a+b+2 c+3 n+1}-q^{2 a+b+c+d+3 n+1}-q^{2 a+b+2 d+3 n+1}+q^{2
 a+2 c+d+4 n+1}+q^{2 a+c+2 d+4 n+1}+$

\noindent $q^{2 a+2 b+2 c+k+4 n+1}+q^{2 a+2 b+c+d+k+4 n}+2 q^{2 a+2 b+c+d+k+4 n+1}+q^{2 a+2 b+c+d+k+4 n+2}+q^{2 a+2 b+2
 d+k+4 n+1}+q^{a+b+2 c+d+5 n+1}+q^{a+b+2 c+d+5 n+2}+q^{a+b+c+2 d+5 n+1}+$

\noindent $q^{a+b+c+2 d+5 n+2}-q^{2 a+b+2 c+d+k+5 n+1}-q^{2 a+b+2 c+d+k+5 n+2}-q^{2
 a+b+c+2 d+k+5 n+1}-q^{2 a+b+c+2 d+k+5 n+2}-q^{a+2 c+2 d+6 n+1}-q^{a+2 c+2 d+6 n+2} - 
 q^{a+2 b+2 c+d+k+6 n+1} - $

\noindent $q^{a+2 b+2 c+d+k+6 n+2}-q^{a+2 b+c+2
 d+k+6 n+1}-q^{a+2 b+c+2 d+k+6 n+2}+q^{2 a+2 c+2 d+k+6 n+2}-q^{2 a+2 b+2 c+d+2 k+6 n}-q^{2 a+2 b+2 c+d+2 k+6 n+1}-q^{2 a+2 b+c+2 d+2 k+6 n} - $

\noindent $q^{2
 a+2 b+c+2 d+2 k+6 n+1}-q^{b+2 c+2 d+7 n+2}-q^{a+b+2 c+2 d+k+7 n}-q^{a+b+2 c+2 d+k+7 n+1}+q^{2 a+b+2 c+2 d+2 k+7 n+1}+q^{2 b+2 c+2 d+k+8
 n+2}+q^{a+2 b+2 c+2 d+2 k+8 n+1} + $

\noindent $ q^{2 a+2 b+2 c+2 d+3 k+8 n}+q^{a+b+3 c+2 d+2 k+9 n}+q^{a+b+3 c+2 d+2 k+9 n+1} + $

\noindent $q^{a+b+3 c+2 d+2 k+9
 n+2}+q^{a+b+2 c+3 d+2 k+9 n}+q^{a+b+2 c+3 d+2 k+9 n+1} + $

\noindent $q^{a+b+2 c+3 d+2 k+9 n+2}+q^{3 c+3 d+k+10 n+1}+q^{3 c+3 d+k+10 n+2}-q^{a+3 c+3 d+2 k+10
 n+2}-q^{b+3 c+3 d+2 k+11 n+2}-q^{a+b+3 c+3 d+3 k+11 n}-q^{a+b+3 c+3 d+3 k+11 n+1}-q^{4 c+3 d+2 k+12 n+1}-q^{3 c+4 d+2 k+12 n+1}+q^{4 c+4 d+3
 k+14 n+1}\big). $

 \ 

\noindent By then setting $\widetilde{G}(n,k;q) = \widetilde{R}(n,k;q) \widetilde{F}(n,k;q)$, 
 we find that the $q$-WZ difference equation in 
 \eqref{qWZequation} holds. 

 If a numerator $q$-shifted factorial terminates, then the relevant series is finite and the required convergence is immediate. Thus, the 
 following ratio-test argument concerns
 only the nonterminating case. 
In the nonterminating case, direct ratio estimates from the displayed formulas for
$\widetilde{F}$ and $\widetilde{G}=\widetilde{R}\widetilde{F}$ show, for every fixed
admissible $(a,b,c,d)$, that $\widetilde{F}(N,k;q)\to0$ as $N\to\infty$ for each
fixed $k$, that $\sum_{k=0}^{\infty}|\widetilde{F}(0,k;q)|$ converges and
$\sum_{n=0}^{\infty}|\widetilde{G}(n,k;q)|$ converges for each fixed $k$, and that
there exist $k_{0}\geq0$, $M>0$, $\rho\in(0,1)$, and a nonnegative sequence
$(A_n)_{n\geq0}$ with $\sum_{n=0}^{\infty}A_n<\infty$ such that
$|\widetilde{G}(n,k;q)|\leq M\rho^{\,k-k_{0}}A_n$ for all $n\geq0$ and
$k\geq k_{0}$; hence all the hypotheses of
Lemma~\ref{lemmaRectangularTelescope} are satisfied.

Therefore, Lemma \ref{lemmaRectangularTelescope} yields the stated $q$-series identity. 
 In addition, we have that $\widetilde{F}(0,k;q)$ is the $k^{\text{th}}$ summand of
 the displayed ${}_{3}\phi_{2}$-series, while
 $(1-q^c)(1-q^d) \widetilde{G}(n,0;q)$ simplifies to the $n^{\text{th}}$ summand 
 on the right-hand side of Theorem \ref{theorem1over27}. 
 \end{proof}

\begin{example}
 Setting $(a, b, c, d) = \big( \frac{1}{2}, \frac{1}{2}, \frac{3}{2}, 1 \big)$ 
 in Theorem \ref{theorem1over27}, 
 and by writing

 \ 

\noindent $p(n) = 1-q^{2 n+2}+q^{4 n+1}+q^{4 n+2}+q^{4 n+3}+q^{6 n+2}-q^{6 n+5}+q^{8 n+3}-q^{10 n+6}-q^{10 n+7}-q^{12 n+7}-q^{14 n+8}$, 

 \ 

\noindent we obtain the $q$-identity 
\begin{multline*} 
 q(1-q^5) \sum_{k=0}^{\infty} 
 q^{2k} \left[ \begin{matrix} 
 q, q \vspace{1mm} \\ 
 q^2, q^3 \end{matrix} \, \Bigg| \, q^2 \right]_{k} = \\ 
 \sum_{n=0}^{\infty} (-1)^n q^{-n-n^2} 
 \frac{ (q;q^2)_{n}^{2} (q^2;q^2)_{n} }{ (q^3;q^6)_{n} (q^7;q^6)_{n} (q^{11};q^6)_{n} } 
 \frac{p(n)}{q^{2 n+1}+q^{4
 n+2}+1}. 
\end{multline*}
 This provides a $q$-analogue of a formula due to Chu and Zhang, since, by 
 letting $q > 1$ and by multiplying both sides of the above $q$-identity by $\frac{3}{4(1-q)}$, 
 and by then letting $q \to 1^{+}$, 
 we obtain 
 the formula 
 $$ \frac{15 \pi}{8} = \sum_{k=0}^{\infty} 
 \left( -\frac{1}{27} \right)^{k} \left[ \begin{matrix} 
 1, \frac{1}{2} \vspace{1mm} \\ 
 \frac{7}{6}, \frac{11}{6} \end{matrix} \right]_{k} (7k+6) $$
 due to Chu and Zhang \cite[Example 27]{ChuZhang2014}. 
\end{example}

\begin{example}
 Setting $(a, b, c, d) = \big( 1, 1, 2, 2 \big)$ 
 in Theorem \ref{theorem1over27}, and by writing 
 
 \ 

\noindent $ r(n) = (-q^{2 n+1}-q^{2 n+2}-3 q^{3 n+2}-q^{4 n+3}+q^{5 n+3}+2 q^{6 n+4}+q^{6 n+5}+2 q^{7 n+5}+q^{8 n+6} - 
 1) / (\left(q^{n+1}+1\right)
 \left(q^{n+1}+q^{2 n+2}+1\right)), $ 

 \ 

\noindent we obtain the $q$-identity 
\begin{multline*}
 -q(1-q)^{3} (1+q) \sum_{k=0}^{\infty} \frac{q^{k}}{(1-q^{k+1})^{2}} = \\ 
 \sum_{n=0}^{\infty} (-1)^{n} q^{-\frac{n(n+1)}{2}}
 \frac{ \left( q;q \right)_{n}^{3} \left( q;q^2 \right)_{n} }{ 
 \left( q^{3};q^{2} \right)_{n} \left( q^{3};q^{3} \right)_{n} \left( q^{4};q^{3} \right)_{n} 
 \left( q^{5};q^{3} \right)_{n} } r(n). 
\end{multline*}
 This is a $q$-analogue of a formula due to Chu and Zhang, in the sense that:
 By letting $q > 1$, and by multiplying both sides of the above $q$-identity 
 by $\frac{3}{2(q-1)}$, 
 and by then letting $q \to 1^{+}$, 
 this produces the formula 
 $$ \frac{\pi^2}{2} = \sum_{k=0}^{\infty} 
 \left( -\frac{1}{27} \right)^{k} \left[ \begin{matrix} 
 1, 1 \vspace{1mm} \\ 
 \frac{4}{3}, \frac{5}{3} \end{matrix} \right]_{k} \frac{7k+5}{2k+1} $$
 due to Chu and Zhang \cite[Example 25]{ChuZhang2014}. 
\end{example}

\begin{example}
 Setting $(a, b, c, d) = \big( 1, \frac{1}{2}, \frac{3}{2}, \frac{3}{2} \big)$ 
 in Theorem \ref{theorem1over27}, and by writing 
 
 \ 

\noindent $ p(n) = q^{2 n+1}-2 q^{2 n+2}+q^{4 n+1}+q^{4 n+2}-q^{4 n+3}+q^{4 n+4}+q^{6 n+2}+2 q^{6 n+3}-q^{6 n+4}-q^{6 n+5}+q^{8 n+3}-q^{8 n+4}-2 q^{8 n+5}-q^{8
 n+6}+q^{8 n+7}-q^{10 n+5}-q^{10 n+6}-q^{10 n+7}+q^{10 n+8}-q^{12 n+6}-q^{12 n+7}-q^{12 n+8}+q^{12 n+9}-q^{14 n+8}+q^{14 n+9}+q^{14
 n+10}+q^{16 n+10}+q^{18 n+11}+1, $ 

 \ 

\noindent we obtain the $q$-identity 
\begin{multline*}
 q(1-q^5) \sum_{k=0}^{\infty} 
 \frac{q^{2k}}{1-q^{2k+1}} \left[ \begin{matrix} 
 q^{2} \vspace{1mm} \\ 
 q^{3} \end{matrix} \, \Bigg| \, q^2 \right]_{k} = 
 \sum_{n=0}^{\infty} (-1)^{n} q^{-n^2} 
 \times \\ \frac{ \left( q;q^2 \right)_{n} \left( q^2;q^2 \right)_{n}^{2} }{ 
 \left( q^{3};q^{6} \right)_{n} 
 \left( q^{7}; q^{6} \right)_{n} \left( q^{11};q^{6} \right)_{n} } 
 \frac{p(n)}{ \left( 1-q^{4n+1} \right) \left( 1 - q^{4n+3} \right) 
 \left( 1 + q^{2n+1} + q^{4n+2} \right) }. 
\end{multline*}
 This provides a $q$-analogue of an identity due to Chu and Zhang, in the sense that: By letting $q > 1$, and by multiplying both sides 
 of the above $q$-identity by 
 3, and by then letting $q \to 1^{+}$, 
 this produces 
 the formula 
 $$ 30 G = \sum_{k=0}^{\infty} 
 \left( -\frac{1}{27} \right)^{k} \left[ \begin{matrix} 
 1, 1 \vspace{1mm} \\ 
 \frac{7}{6}, \frac{11}{6} \end{matrix} \right]_{k} \frac{112k^2+192k+83}{(4k+1)(4k+3)} $$
 due to Chu and Zhang \cite[Example 29]{ChuZhang2014}. 
\end{example}

\section{Conclusion}
 Our multiparameter {\tt EKHAD}-normalization method may be applied much more broadly, using $q$-analogues of hypergeometric 
 input functions as in our past work \cite{Campbell2025Hypergeometric}. For brevity, and in view of the many $q$-analogues of 
 previously published results from many authors highlighted 
 in Section \ref{sectionMotivating}, we leave this, along with the research area below, to separate projects. 

 We encourage the exploration of alternatives to and variants of our $q$-WZ-based techniques given above.  For example, the  
 $q$-WZ-based technique employed in our past work \cite{Campbell2025further} to produce the final $q$-analogue of the  Fabry--Guillera 
 formula listed in Section \ref{subFabryGuillera} differs in form 
 from the $q$-WZ-based techniques developed in the present paper, 
 since the former produces series whose summands contain $q$-Pochhammer symbols indexed by $\infty$. 
 For example, by adapting the approach from our past research on a $q$-version of the Fabry--Guillera formula 
 \cite{Campbell2025further} and by using a recent terminating $q$-binomial identity due to 
 Berkovich in 2024 \cite{Berkovich2024}, it can be shown that 
\begin{multline*}
 \sum_{n=0}^{\infty} \frac{q^{n(n+1)}}{3^n} 
 \left( \frac{1}{ \left( -3q;q^2 \right)_{\infty} \left( -\frac{q}{3}; q^2 \right)_{\infty} } 
 - \frac{ \left( q^{-2n+1};q^2 \right)_{\infty} \left( q^{2n+3};q^2 \right)_{\infty} }{ \left( q^2;q^2 \right)_{\infty} } 
 \right) = \\ 
 - 6 q \sum_{n=0}^{\infty} \frac{q^{2n}}{(q^{2n+1} + 3) (3 q^{2n+1} +1) } \frac{ 
 \left( q^{2n+1};q^2 \right)_{\infty} \left( q^{2n+3};q^2 \right)_{\infty} }{ \left( q^{4n+2};q^2 \right)_{\infty} 
 \left( -3q;q^2 \right)_{n} \left( -\frac{q}{3};q^2 \right)_{n}}, 
\end{multline*}
 for $q \in (0, 1)$, 
 and this gives us a $q$-analogue of the $\pi$-formula
 $\frac{2 \pi}{\sqrt{3}} = \sum_{n=1}^{\infty} \frac{3^n}{n \binom{2n}{n}}$. 
 
\subsection*{Acknowledgements}
 The author is very thankful to Yan-Ping Mu for much help formulating the proof of Theorem \ref{theoremquarter},  and thanks Kam 
 Cheong Au and Michael Schlosser for useful feedback.  The author is very thankful to a number of anonymous reviewers for  very    detailed  
  and very helpful reports that have led to many improvements to this research paper. 
 GPT-5.6 Pro was used to be of assistance with light editing and 
 with some revisions of a previous and finalized version of 
 this research paper, including revisions related to the formulations and proofs of Lemmas 
 \ref{lemmaEKHAD} and \ref{lemmaRectangularTelescope}. 
  GPT-6 Astra Pro was also used in a subsequent review of the manuscript to help identify some 
  errors in displayed mathematical expressions and check the resulting corrections. 
  All AI-generated suggestions incorporated into the manuscript were reviewed, revised where appropriate, and independently verified
 by the author, who assumes full responsibility for the mathematical content.

 \ 

{\textsc{John M. Campbell}} 

\vspace{0.1in}

Department of Mathematics and Statistics

Dalhousie University 

6283 Alumni Crescent, Halifax, NS B3H 4R2

\vspace{0.1in}

{\tt jh241966@dal.ca}

\end{document}